\documentclass{cicp}
\usepackage{amsmath}
\usepackage{amsthm}
\usepackage{amsfonts}

\usepackage{algorithm}
\DeclareMathOperator{\argmin}{argmin}
\usepackage{comment}
\usepackage{graphicx}
\usepackage{indentfirst}
\usepackage[colorlinks=true, allcolors=blue]{hyperref}
\usepackage{algpseudocode}
\usepackage{multirow}
\usepackage{overpic}
\usepackage{caption}

\usepackage[version = 4]{mhchem}

\usepackage{subcaption}

\usepackage{xcolor}
\usepackage{float}

\usepackage[normalem]{ulem}

\usepackage{bm}

\newcommand\dd{\mathrm{d}}
\newcommand\pp{\partial}

\newcommand\x{\bm{x}}
\newcommand\uvec{\mathbf{u}}
\newcommand\X{ \bm{X} }

\newcommand\n{{\bf n}}

\newcommand\F{\mathsf{F}}

\numberwithin{equation}{section}
\newtheorem{prop}{Proposition}[section]
\DeclareMathAlphabet\mathbfcal{OMS}{cmsy}{b}{n}  

\begin{document}
\title{On a Modified Random Genetic Drift Model: Derivation and a Structure-Preserving Operator-Splitting Discretization}

\author[Chen C A et.~al.]{Chi-An Chen\affil{1},
       Chun Liu\affil{1}, and Yiwei Wang\affil{2}\comma\corrauth}
\address{\affilnum{1}\ Department of Applied Mathematics, Illinois Institute of Technology, Chicago, IL 60616, USA.  \\
           \affilnum{2}\ Department of Mathematics, University of California Riverside, Riverside, CA, 92521, USA. 
           }
\emails{{\tt cchen156@hawk.iit.edu} (Chi-An Chen), {\tt cliu124@iit.edu} (Chun Liu),
          {\tt yiweiw@ucr.edu} (Yiwei Wang)}

%
%
\begin{abstract} 
One of the fundamental mathematical models for studying random genetic drift is the Kimura equation, derived as the large-population limit of the discrete Wright-Fisher model. However, due to the degeneracy of the diffusion coefficient, it is impossible to impose a suitable boundary condition that ensures the Kimura equation admits a classical solution while preserving biological significance. In this work, we propose a modified model for random genetic drift that admits classical solutions by modifying the domain of the Kimura equation from $(0, 1)$ to $(\delta, 1 - \delta)$ with $\delta$ being a small parameter, which allows us to impose a Robin-type boundary condition. By introducing two additional variables for the probabilities in the boundary region, we effectively capture the conservation of mass and the fixation dynamics in the original model. To numerically investigate the modified model, we develop a hybrid Eulerian-Lagrangian operator splitting scheme. The scheme first solves the flow map equation in the bulk region using a Lagrangian approach with a no-flux boundary condition, followed by handling the boundary dynamics in Eulerian coordinates. This hybrid scheme ensures mass conservation, maintains positivity, and preserves the first moment. Various numerical tests demonstrate the efficiency, accuracy, and structure-preserving properties of the proposed scheme. Numerical results demonstrate the key qualitative features of the original Kimura equation, including the fixation behavior and the correct stationary distribution in the small-$\delta$ limit.
\end{abstract}


\maketitle

\section{Introduction}
Genetic drift describes random changes in allele frequencies within a finite population across generations. This evolutionary process can be mathematically modeled as a stochastic process \cite{ewens2004mathematical}, known as the Wright-Fisher model, which was introduced by Fisher \cite{fisher1923xxi} and Wright \cite{wright1929evolution, wright1937distribution, wright1945differential}. The Wright-Fisher model expresses the random genetic drift as a discrete-time Markov chain. Specifically, consider a population with fixed finite size $N$ and two alleles $A_{1}$ and $A_{2}$, and let $X_{t}$ denote the proportion of genes $A_{1}$ in the t-th generation.  Assuming that the alleles in generation $t+1$ are obtained by random sampling (with replacement) from generation $t$, without mutation, migration, or selection, the transition probability is given by 
\begin{equation}
\label{Simple_Wright_Fisher}
P \left(X _{t+1} = \frac{k}{2N} \Big| X_{t} = \frac{n}{2N} \right) = \binom{2N}{k} \left(\frac{n}{2N} \right)^{k} \left(1 - \frac{n}{2N} \right)^{2N - k}, \ k,n = 0,1, \cdots, 2N.
\end{equation}

Let $\rho_{t,n}$ denote the probability density of the gene frequency $X_{t}$. If the population size $N$ is large enough, $X_{t}$ and $\rho_{t,n}$ can be approximated by the continuous gene frequency $x(t)$ and a distribution $\rho(x,t)$ respectively. In the above case that mutation, migration, and selection, the probability density $\rho(x,t)$ satisfies the following diffusion equation   
\begin{equation}
    \label{Kimura equation}
    \frac{\partial \rho(x,t)}{\partial t} = \frac{1}{4N}\frac{\partial^{2} }{\partial x^{2}}(x(1 - x)\rho(x,t)), \ x \in (0,1), \ t > 0.
\end{equation}
Equation \eqref{Kimura equation} is known as the Kimura equation \cite{kimura1964diffusion, kimura1954stochastic, zhao2013complete}.
By rescaling the time, the Kimura equation \eqref{Kimura equation} can be written as 
\begin{equation}
    \label{Original_Kimura_eq}
    \frac{\partial \rho(x,t)}{\partial t} = \frac{\partial^{2} }{\partial x^{2}}(x(1 - x)\rho(x,t)), \ x \in (0,1), \ t > 0. 
\end{equation}
See the Appendix for a detailed derivation of the Kimura equation from the stochastic process.

Although \eqref{Original_Kimura_eq} is a linear PDE of $\rho(x, t)$, the boundary conditions and space of solution of (\ref{Original_Kimura_eq}) are unclear, due to the degeneracy of the diffusion coefficient  $x(1 - x)$ at the boundary $x = 0$ and $x = 1$ \cite{epstein2010wright, epstein2013degenerate, feller1951diffusion}. As $\rho(x, t)$ represents a probability density function, it must satisfy mass conservation
 \begin{equation}
   \int_{0}^1 \rho(x, t) \dd x = 1 \ ,
 \end{equation}
 which leads to a no-flux boundary, given by
 \begin{equation}
  \pp_x(x(1 - x) \rho) = 0 \ , \quad x = 0 ~\text{or}~ 1 \ , \quad \forall t \ .
 \end{equation}
 However, such a boundary condition excludes the existence of a classical solution of \eqref{Original_Kimura_eq}. In \cite{mckane2007singular, chalub2009non}, the authors prove that for a given $\rho_0 \in \mathcal{BM}^+([0, 1])$, there exists a unique solution to \eqref{Original_Kimura_eq} with $\rho(x, t) \in L^{\infty} ([0, \infty), \mathcal{BM}^+([0, 1]))$, and the solution $\rho(x, t)$ can be expressed as
\begin{equation}\label{measure_solution}
 \rho(x, t) = q(x, t) + a(t) \delta_0  + b(t) \delta_1 \ .
 \end{equation}
 Here, $\mathcal{BM}^+([0, 1])$ is the space of all (positive) Radon measures on $[0,1]$, $\delta_0$ and  $\delta_1$ are Dirac delta functions at 0 and 1 respectively, and $q(x, t) \in C^{\infty} (\mathbb{R}^+; C^{\infty}([0, 1]))$ is a classical solution to \eqref{Original_Kimura_eq}. Moreover, it is proved that \cite{chalub2009non}, as $t \rightarrow \infty$, $q (x, t) \rightarrow 0$ uniformly, and 
 $a(t)$ and $b(t)$ are monotonically increasing functions such that  
\begin{equation}\label{eq_a_b}
 \begin{aligned}
   & a^{\infty} = \lim_{t \rightarrow \infty } a(t) = \int_{0}^1 (1 - x) \rho_0(x) \dd x \ ,\\
   & b^{\infty} = \lim_{t \rightarrow \infty } b(t) = \int_{0}^1 x \rho_0(x) \dd x \  . \\
 \end{aligned}
 \end{equation}
The equilibrium (\ref{eq_a_b}) is determined through another biological requirement, conservation of the fixation
probability, i.e.,
\begin{equation}
 \frac{\dd}{\dd t} \int \psi(x) \rho(x, t) = 0,
\end{equation}
where $\psi(x)$ is the fixation probability function that satisfies 
$$ x(1- x) \psi''= 0, \quad \psi(0) = 0, \quad \psi(1) = 1\ .$$
 The fixation probability function  describes the probability of allele $A_1$ fixing in a population while allele $A_2$ goes extinct, under the condition of starting from an initial composition of $x$. For the pure drift case considered in the paper, $\psi(x) = x$. Hence, the model conserves the first moment.

The absence of a classical solution to the Kimura equation poses significant challenges for both its theoretical analysis and numerical approximation \cite{zhao2013complete, mckane2007singular, chalub2009non, xu2019behavior, duan2019numerical, carrillo2010numerical, dangerfield2012boundary, jenkins2017exact}. For example, although many numerical methods have been proposed to solve the Kimura equation \cite{zhao2013complete, xu2019behavior, duan2019numerical, carrillo2010numerical, dangerfield2012boundary, jenkins2017exact}, it is difficult to capture the delta-function type singular behavior of the original model. To retain the key characteristics of the original model while ensuring the existence of a classical solution, in \cite{liu2023continuum}, the authors proposed a new continuum model for random genetic drift by modifying the domain to $(\delta, 1-\delta)$ and introducing dynamical boundary conditions to handle the fixation dynamics. The modified model is given by 
\begin{equation}
  \label{Regularized_Kimura_equation}
  \begin{cases}
    & \partial_{t} \rho  = \pp_{xx}^2 ( x(1 - x) \rho), \quad x \in (\delta, 1 - \delta), t >0 \\
    &  \pp_x (x (1 - x) \rho) |_{x = \delta} = a'(t) \\
    &  \pp_x (x (1 - x) \rho) |_{x = 1 - \delta} = - b'(t) \\
    & a'(t) =    -((\epsilon a) - \alpha \rho(\delta , t))   \\
    & b'(t) =    -((\epsilon b) - \alpha \rho(1 - \delta , t)),  \\
  \end{cases}
 \end{equation} 
where $\delta > 0$ and $\epsilon > 0$ are artificial small parameters, and $\alpha > 0$ is an additional parameter. The new continuum model is based on the idea of introducing the surface densities $a(t)$ and $b(t)$ on the boundary and modeling the interaction between the bulk and surface densities as a chemical reaction \cite{liu2023continuum, wang2022some, knopf2021phase}. The small parameter $\epsilon$ represents the reverse reaction rate, which describes the rate at which the mass on the boundary transitions back into the bulk domain. The functions $a(t)$ and $b(t)$ can be seen as approximated delta functions developed at the boundary if we treat the density at $x \in [0, \delta]$ as $\frac{a(t)}{\delta}$ and at $x \in [1 - \delta, 1]$ as $\frac{b(t)}{\delta}$, which are rectangular functions used to approximate delta functions. In \cite{liu2023continuum}, the authors prove that the regularized Kimura equation \eqref{Regularized_Kimura_equation} admits a classical solution for fixed $\delta$ and $\epsilon$, and numerically demonstrate that the new model captures key features of random genetic drift, such as gene fixation and conservation of the first moment. However, rigorous proofs of the convergence of solutions of the regularized model to those of the original model as $\delta$ and $\epsilon$ approach zero remain open.

The purpose of this work is to numerically study the regularized Kimura equation (\ref{Regularized_Kimura_equation}) with $\epsilon = 0$, given by
\begin{equation}\label{Eq_epsilon_0}
 \begin{cases}
  &  \pp_t \rho = \pp_{xx}^2 ( x(1 - x) \rho), \quad x \in (\delta, 1 - \delta), t >0 \\ 
  & \pp_x (x (1 - x) \rho) |_{x = \delta} = \alpha \rho (\delta, t) \\ 
  &  \pp_x (x (1 - x) \rho)  |_{x = 1 - \delta} =  - \alpha \rho(1 - \delta, t) \\
 \end{cases}
\end{equation}
along with the ODE system imposed at the boundary: 
\begin{equation} \label{boundary_ODE_1}
 a'(t) =  \alpha\rho(\delta, t), \quad b'(t) = \alpha\rho(1 - \delta, t).
\end{equation}
As noted in \cite{liu2023continuum}, the variational structure of (\ref{Regularized_Kimura_equation}) is absent due to the irreversibility of the chemical reaction governing mass exchange between the bulk and the boundary. 
The loss of the variational structure complicates analysis and computation. 
To address the difficulty in numerical calculation, a hybrid operator splitting scheme is introduced, where the bulk region is handled using a Lagrangian numerical scheme, while the boundary conditions are treated with an Eulerian scheme.
To overcome the computational challenges, we introduce a hybrid operator-splitting scheme: the bulk dynamics are handled using a Lagrangian method, while the boundary conditions are treated with an Eulerian approach.
Numerical tests show that the key properties of the Kimura equation are accurately captured.Moreover, we compare our numerical results with those obtained using the Eulerian scheme in \cite{liu2023continuum} and the comparison shows good agreement, with the advantage of requiring fewer particles.

The rest of the paper is organized as follows.
In Section 2, we present an overview of the energetic variational approach (EnVarA), which will be used to derive the modified Kimura equation.
In Section 3, we apply EnVarA to derive the modified Kimura equation and establish several important properties of the modified system.
In Section 4, we propose an operator-splitting numerical scheme for the truncated Kimura equation and introduce two numerical methods for solving the optimization problem \eqref{optimization}, which arises from the discretized force balance equation.
In Section 5, we conduct numerical studies of the system \eqref{Eq_epsilon_0}–\eqref{boundary_ODE_1}. 
In Section 6, we provide concluding remarks on the new model and numerical scheme.
Finally, a review of the derivation of the original model is included in the appendix.


\section{Preliminary}

In this section, we introduce the energetic variational approach (EnVarA) as a framework for establishing the variational formulation of the Kimura equation and deriving its regularized version. 

The EnVarA, motivated by the pioneering work of Rayleigh \cite{strutt1871some}, and Onsager \cite{onsager1931reciprocal, onsager1931reciprocal2} for nonequilibrium thermodynamics,
is based on the first and second laws of thermodynamics.  
The first law of thermodynamics states that the rate of change of the total energy, which is the sum of kinetic energy $\mathcal{K}$ and internal energy $\mathcal{U}$, within a system $\mathcal{P}$ is equal to the work done on $\mathcal{P}$ and the heat transferred to $\mathcal{P}$: 
\begin{equation}
    \frac{d}{dt} (\mathcal{K} + \mathcal{U}) =  \dot{\mathcal{W}} + \dot{\mathcal{Q}}
\end{equation}
The second law relates the heat transfer with the entropy $\mathcal{S}$:
\begin{equation}
    T\frac{dS}{dt} = \dot{\mathcal{Q}} + \Delta 
\end{equation} 
where $T$ is the temperature, and $\Delta \geq 0$ represents the rate of entropy production. For an isothermal and mechanically closed system, where the temperature is constant and no work is performed (i.e., $\dot{\mathcal{W}} = 0$), the energy dissipation law can be obtained by subtracting the second law from the first: 
\begin{equation}
    \frac{d}{dt} (\mathcal{K} + \mathcal{U} - TS) = - \Delta ,
\end{equation}
where $\mathcal{F} = \mathcal{U} - TS$ is the Helmholtz free energy. By denoting  $E^{total} = \mathcal{K} + \mathcal{F}$, the energy dissipation law can be rewritten as  
\begin{equation}\label{ED_1}
    \frac{d}{dt}E^{total} = -   \Delta.   
\end{equation}
From the energy-dissipation law (\ref{ED_1}), the EnVarA derives the dynamics through two distinct variational processes: the Least Action Principle (LAP) and the Maximum Dissipation Principle (MDP).

\subsection{EnVarA for Continuum Mechanics} 


In the context of continuum mechanics, the primary variable in this variational framework is the flow map $\x(\X, t)$. Here, $\X$ is the Lagrangian coordinate (original labeling) of the particle, and $\x$ is the Eulerian coordinate. For a given (smooth) velocity field $\uvec(\x, t)$, the flow map is defined by the ordinary differential equation
\begin{equation}
\begin{cases}
\frac{d}{dt} \x(\X,t) = \mathbf{u}(\x(\X,t),t), \\
\x(\X,0) = \X,
\end{cases}
\end{equation} 

In a conservative system, the LAP \cite{arnol2013mathematical} states that the dynamics of the system can be derived from the variation of the action functional $\mathcal{A}(\textbf{x}) 
= \int_{0}^{T}(\mathcal{K} - \mathcal{F} ) dt $ with respect to the flow map $\x(\X,t) $, which implies that 
\begin{equation*}
    0 = \frac{d}{d\epsilon} \Bigg|_{\epsilon = 0}\mathcal{A}(\x + \epsilon \mathbf{y}) = \int_{0}^{T} \frac{\delta \mathcal{A}}{\delta \x} \cdot \mathbf{y} dt,
\end{equation*}
where $\x \in \mathcal{M}$, and $\mathbf{y}(\X,t)$ is any test function such that $\x + \epsilon \mathbf{y} \in \mathcal{M}$, with $\mathcal{M}$ being a prescribed admissible set. The conservative force can be obtained from the variation of the action functional 
\begin{equation*}
    F_{con} = \frac{\delta \mathcal{A}}{\delta \x}.
\end{equation*}
For the dissipation part, the MDP states that the dissipative force can be obtained by taking the variation of the dissipation potential $\mathcal{D}$, which equals $\frac{1}{2}\triangle$ in the linear response region  \cite{onsager1931reciprocal, onsager1931reciprocal2}, with respect to $\textbf{x}_{t}$
\begin{equation}\label{diss_f}
    F_{dis} = \frac{\delta \mathcal{D}}{\delta \textbf{x}_{t}}. 
\end{equation}
Finally, in accordance with the force balance, we have 
\begin{equation}
    \frac{\delta \mathcal{A}}{\delta \textbf{x}} = \frac{\delta \mathcal{D}}{\delta \textbf{x}_{t} }\ ,
\end{equation}
which is the dynamics of the system.

For continuum mechanical systems, the evolution of physical variables, such as the density function, is determined by the evolution of the flow map $\x(\X, t)$ through kinematics.
To determine the value of physical variables at each material point, one needs the deformation tensor, which is defined by
\begin{equation}\label{def_F}
     \tilde{\F}(\x(\X, t),  t) = \F(\X, t) =   \nabla_{\X} \x(\X, t)  
\end{equation}
For a mass density $\rho(\x,t)$, let $\rho_{0}(\X)$ be the initial density. Then the mass conservation means
\begin{equation}\label{density_Lagrangian}
    \rho(\x(\X,t)) = \frac{\rho_{0}(\X)}{\det \F(\X, t)}\ ,
\end{equation} 
which is equivalent to the continuity equation 
\begin{equation}\label{transport_equation}
    \rho_{t} + \nabla \cdot (\rho \mathbf{u}) = 0,
\end{equation} 
in Eulerian coordinates.


To illustrate the general framework of EnVarA, we show how a generalized diffusion can be derived from an energy-dissipation law using EnVarA. 
Generalized diffusion describes the evolution of a conserved quantity $\rho$ that satisfies the transport equation \eqref{transport_equation}. Its dynamics are governed by the following energy-dissipation law \cite{giga2017variational}: \begin{equation} \frac{d}{dt} \int_{\Omega} \omega(\rho) + \rho V(\x) \dd \x = - \int_{\Omega} \eta(\x,\rho) |\uvec|^2  \dd \x, \end{equation} 
where $\omega(\rho)$ is the internal energy density, $V(\x)$ is an external potential, and $\eta(\x,\rho)$ represents a possibly inhomogeneous mobility. 
Due to the kinematics (\ref{density_Lagrangian}), the free energy can be reformulated as a functional of $\x(\X, t)$ in Lagrangian coordinates, so is the action functional $\mathcal{A}$.
A direct computation shows that
\begin{equation*}
  \begin{aligned}
    & \delta \mathcal{A} = - \delta \int_{0}^T  \int_{\Omega_0} \omega(\rho_0(\X)/ \det \F) \det \F + \rho_0(\X) V(\x(\X, t)) \, \dd \X \dd t \\
    & = - \int_{0}^T \int_{\Omega_0} \left( - \omega'  \left( \frac{\rho_0(\X)}{\det \F} \right)  \cdot \frac{\rho_0(\X)}{\det \F} + \omega\left(\frac{\rho_0(\X)}{\det \F}\right) \right)  \times  (\F^{-\rm{T}} : \nabla_{\X} \delta \x)\det F  + \rho_0(\X) \nabla V \cdot \delta \x \  \dd \X \dd t, \\
      \end{aligned}
\end{equation*}
Here, $\Omega_0 = \Omega$ is the reference domain, and $\delta \x(\X, t)$ is the test function satisfying $\tilde{\delta \x} \cdot \n = 0$ with $\n$ being the outer normal of $\Omega$ in Eulerian coordinates, where $\tilde{\delta \x}(\x(\X, t), t) = \delta \x(\X, t)$ and $\delta(\X, t)$ without ambiguity. Pushing forward to Eulerian coordinates, we have
\begin{equation}\label{LAP1}
  \begin{aligned}
\delta \mathcal{A} & = - \int_{0}^T \int_{\Omega} ( -  \omega'(\rho)  \rho + \omega) \nabla \cdot (\tilde{\delta \x}) + \rho \nabla V \cdot \tilde{\delta \x} \dd \x = 
 - \int_{0}^T \int_{\Omega} \nabla  ( \omega'(\rho) \rho - \omega + V(\x)) \cdot \tilde{\delta \x}  \dd \x \dd t, \\
  \end{aligned} 
\end{equation}
which indicates that $$\frac{\delta \mathcal{A}}{\delta \x} = - \nabla (\omega'(\rho)  \rho - \omega) = - \rho \nabla \mu,$$ where $\mu = \frac{\delta \mathcal{F}}{\delta \rho} = \omega'(\rho) + V(x)$ is the chemical potential. 

For the dissipation part, since $\mathcal{D} = \frac{1}{2} \int  \eta(\x, \rho) |\uvec|^2  \dd \x$ it is easy to compute that $\frac{\delta \mathcal{D}}{\delta \uvec} = \eta(\x, \rho) \uvec$. As a consequence, we have the force balance equation
\begin{equation}\label{FB1}
  \eta(\x, \rho) \uvec =  - \rho \,\nabla \mu.
\end{equation}
Combining the force balance equation (\ref{FB1}) with the kinematics (\ref{transport_equation}), one can obtain a generalized diffusion equation
\begin{equation}\label{G_diffusion}
\rho_t = \nabla \cdot \left( \frac{\rho^2}{\eta(\rho)}   \nabla \mu \right).
\end{equation}

Formally, the original Kimura equation can be viewed as a generalized diffusion by taking $V(x) = \log(x(1-x))$ and $\eta(x,\rho) = \frac{\rho}{x(1-x)}$: 
\begin{equation}\label{1}
    \dfrac{d}{dt}\int_{0}^{1}\rho(x,t)\log(x(1-x)\rho(x,t)) \ dx = - \int_{0}^{1} \dfrac{\rho(x,t)}{x(1-x)} |u|^{2} \ dx,
\end{equation} 
The corresponding force balance equation is
\begin{equation}
    \frac{1}{x(1 -x)} \rho u = - \rho \pp_x (\ln \rho + \ln (x (1 - x))).
\end{equation}
However, the derivation is formal as both $V(x)$ and $\eta(x, \rho)$ blow up at $x = 0$ and $x = 1$. So, it is crucial to change the domain from $(0, 1)$ to $(\delta, 1 - \delta)$ such that the energy-dissipation law (\ref{1}) is well-defined.



\subsection{EnVarA for Chemical Reactions}
A key component of the modified Kimura equation in \cite{liu2023continuum} is the dynamical boundary condition introduced to describe the fixation dynamics on $x = 0$ and $x = 1$ after altering the domain to $(\delta, 1 - \delta)$. Since the dynamical boundary condition can be interpreted as a chemical reaction \cite{wang2022some}, we briefly review in this subsection how reaction kinetics can be modeled using EnVarA.

Consider a reversible chemical reaction system involving two species $\{A, B\}$ and a reaction 
\begin{equation}
    \alpha A \rightleftharpoons \beta B.
\end{equation} Let $c_{A}, c_{B} \in \mathbb{R}_{+}$ be the concentrations of the species A and B, respectively.  

To derive the reaction kinetics using the EnVarA, we introduce the \emph{reaction trajectory} \( R(t) \) as the primary variable in the variational formulation. The reaction trajectory \( R \), which represents the number of forward reactions that have occurred by time \( t \) (and may take negative values), is analogous to the flow map in mechanical systems. The relation between the reaction trajectory and the concentrations of chemical species is given by
\begin{equation}
\begin{cases}
    c_{A}(t) = c_{A}^{0} + \sigma_{A} R(t), \\ 
    c_{B}(t) = c_{B}^{0} + \sigma_{B} R(t),
\end{cases}
\end{equation}
where \( c_{A}^{0} \) and \( c_{B}^{0} \) denote the initial concentrations of species \( A \) and \( B \), respectively, and \( \sigma_{A} \), \( \sigma_{B} \) are the stoichiometric coefficients. For a reaction of the form \( \alpha A \rightarrow \beta B \), we have \( \sigma_{A} = -\alpha \) and \( \sigma_{B} = \beta \).

Using the reaction trajectory, the chemical kinetics can be expressed through the energy-dissipation law in terms of $R$ and $\partial_{t}R$
\begin{equation}\label{chemical_energy_dissipation_law}
    \frac{d}{dt} \mathcal{F}[c_{A}(R), c_{B}(R)] = - \mathcal{D}_{chem}[R, \partial_{t}R].  
\end{equation}
The law of mass action is commonly used and can be derived from the energy-dissipation law \eqref{chemical_energy_dissipation_law} by setting 
\begin{equation}
\mathcal{F}[c_{A}(R), c_{B}(R)] =  c_{A}\left[\log\left(\frac{c_{A}}{c_{A}^{\infty} }\right) - 1\right] + c_{B}\left[\log\left(\frac{c_{B}}{c_{B}^{\infty}}\right) - 1\right],  \quad 
    \mathcal{D}_{chem}[R, \partial_{t}R] = \partial_{t}R\log\left[\frac{\partial_{t}R}{\eta(c_{B}(R))} + 1 \right], 
\end{equation}
where $c_{A}^{\infty}$ and $c_{B}^{\infty}$ represent the detailed balance equilibrium for the two species, and $\eta(c_{B}(R))$ denotes the mobility for the reaction. 
Unlike mechanical processes, chemical reactions typically occur far from equilibrium \cite{de2013non},  which means that the chemical dissipation $\mathcal{D}_{chem}$ is generally not quadratic in terms of $\partial_{t}R$. A general form of chemical dissipation can be expressed by 
\begin{equation}
    \mathcal{D}_{chem}[R, \partial_{t}R] = 
    \left(\Gamma(R, \partial_{t}R), \partial_{t}R\right) =  \Gamma(R, \partial_{t}R) \partial_{t}R\geq 0.
\end{equation} 
The energy-dissipation law \eqref{chemical_energy_dissipation_law} implies 
\begin{equation}
    \Gamma(R, \partial_{t}R) = - \frac{\delta \mathcal{F}}{\delta R}, 
\end{equation}
representing the chemical force balance \cite{wang2020field, wang2022some, liu2021structure}. 
By taking the variation, we obtain the force balance equation given by
\begin{equation}
    \log\left[\frac{\partial_{t}R}{\eta(c_{B}(R))} + 1 \right] = - \frac{\delta }{\delta R} \mathcal{F}[R].
\end{equation}
For more details on the energetic variational approach for chemical reactions, we direct the reader to \cite{wang2020field, wang2022some, liu2021structure}.


%
\section{Modified Kimura equation} 

In this section, we propose the modified Kimura equation, which is obtained
as the limit of the regularized Kimura equation (\ref{Regularized_Kimura_equation}) as $\epsilon \rightarrow 0$. 

We first briefly review the derivation of the regularized Kimura equation (\ref{Regularized_Kimura_equation}) proposed in \cite{liu2023continuum} using the EnVarA. To compensate for singularities at the boundary of the original Kimura equation, the regularized model modifies the domain from $(0,1)$ to $(\delta, 1 - \delta)$, where $\delta > 0$ is a small artificial parameter.  
The function $\rho(x, t)$ represents the probability that the gene frequency is equal to $x \in (\delta, 1 - \delta)$ at time $t$.  The probability that the gene frequency at the boundary regions $[0, \delta)$ and $(1 - \delta, 1]$ are denoted by $a(t) / \delta$ and $b(t) / \delta$, respectively, with $a(t)$ and $b(t)$ being two additional variables.  The interactions between bulk and boundary are viewed as generalized chemical reactions
\[  \rho(\delta, t) \ce{<=>} a(t), \quad  \rho(1 - \delta, t) \ce{<=>} b(t) \]
Hence, $\rho(x, t)$, $a(t)$ and $b(t)$ satisfy the boundary condition
\begin{equation}\label{kinemtics_DB}
 \begin{aligned}
  & \pp_t \rho + \pp_x (\rho u) = 0, \quad x \in (\delta, 1 - \delta) \\
  &  \rho u(\delta, t) = - \dot{R}_0(t), \quad \rho u(1 - \delta, t) = \dot{R}_1 (t) \\
  &  a'(t) = \dot{R}_0(t), \quad    b'(t)  = \dot{R}_1(t), \\
 \end{aligned}
\end{equation}
Here, $R_0(t)$ and $R_1(t)$ denote the reaction trajectory from $x = \delta$ to $x = 0$ and from $x = 1 - \delta$ to $x = 1$ respectively. The kinematics assumption automatically guarantees the mass conservation
\begin{equation}
 \begin{aligned}
   \frac{\dd}{\dd t}  \left(  \int_{\delta}^{1-\delta} \rho(x, t) \dd x + a(t) + b(t) \right) = 0,
 \end{aligned}
\end{equation}

Following the general approach to a dynamical boundary condition \cite{knopf2021phase, wang2022some}, the overall system can be modeled through an energy-dissipation law,
\begin{equation}\label{ED_RK}
\begin{aligned}
 & \frac{\dd}{\dd t} \left[ \int_{\delta}^{1 - \delta} \rho \ln \left( x(1 - x) \rho \right)\dd x + G_0(a) + G_1(b) \right]   = - \int_{\delta}^{1 - \delta} \frac{\rho}{x(1 - x)}  |u|^2   \dd x - \dot{R_{0}}\Psi_0(R_0, \dot{R_0}) - \dot{R_{1}}\Psi_1(R_1, \dot{R_1})
 \end{aligned}
\end{equation}
where $G_0(a)$ and $G_1(b)$ are the free energies on the boundary.  The remaining question is how to choose $G_i (i = 0, 1)$ and $\Psi_i (i = 0, 1)$ to capture the qualitative behavior of the original Kimura equation. As in \cite{liu2023continuum}, we take 
\begin{equation}\label{Form_G}
 G_0(q) = G_1(q) = G(q) =  q \ln (\kappa(\epsilon) \delta (1 - \delta) q   )\ , 
\end{equation}
and
\begin{equation}\label{Form_Gamma}
 \Psi_0(R_{0}, \dot{R}_{0})  =  \ln \left( \frac{\dot{R}_{0}}{\gamma_0 a}  + 1  \right), \quad  \Psi_1(R_{1}, \dot{R}_{1})  = \ln \left( \frac{\dot{R}_{1}}{\gamma_1 b}  + 1  \right) \ .
\end{equation}
Here, $\gamma_0$ and $\gamma_1$ represent the reaction rates from the surface to the bulk.  In our case, we  assume  $\gamma_{0} = \gamma_{1} = \epsilon$. 

By an energetic variational procedure introduced previously, we can obtain the velocity equation
\begin{equation}\label{velocity_1}
 \frac{1}{x(1 - x)}\rho u = - \rho \pp_x ( \ln \rho +  \ln (x (1 - x))) = - \pp_x \rho - \frac{\rho}{x (1 - x)} (1 - 2x), \quad x \in (\delta, 1 - \delta) 
\end{equation}
which can be simplified as
\begin{equation}\label{velocity}
 \rho u = -  \pp_x (x (1 - x) \rho), \quad x \in (\delta, 1 - \delta), 
\end{equation}
and the equations for reaction rates
\begin{equation}\label{Eq_R}
 \begin{aligned}
   & \ln \left( \frac{\dot{R}_0}{\epsilon a} + 1   \right) =  - (\ln (\kappa(\epsilon) a) -  \ln \rho(\delta, t))  \\
   & \ln \left( \frac{\dot{R}_1}{\epsilon b}  + 1  \right) = - (\ln (\kappa(\epsilon) b) - \ln \rho(1 - \delta, t) )  \\
 \end{aligned}
\end{equation}
One can rewrite (\ref{Eq_R}) as
\begin{equation}\label{Eq_RR}
\dot{R}_0 =  \frac{\epsilon}{\kappa(\epsilon)}\rho(\delta, t) - \epsilon a, \quad  \dot{R}_1 = \frac{\epsilon}{\kappa(\epsilon)}\rho(1 - \delta, t) - \epsilon b \ .
\end{equation}
Combining (\ref{velocity}) and (\ref{Eq_RR}) with the kinematics (\ref{kinemtics_DB}), we arrive at the final equation 
\begin{equation}\label{Eq_Final_LMA}
 \begin{cases}
   &  \pp_t \rho = - \pp_x (\rho u), \quad   x \in (\delta, 1 - \delta) \\
   & \rho u = -  \pp_x (x (1 - x) \rho), \quad x \in (\delta, 1 - \delta), \\ %
   & \rho u(\delta, t) = - a'(t), \quad \rho u(1 - \delta, t) = b'(t) \\
   & a'(t) =  \frac{\epsilon}{\kappa(\epsilon)}\rho(\delta, t) - \epsilon a  \\
   & b'(t) =  \frac{\epsilon}{\kappa(\epsilon)}\rho(1 - \delta, t) - \epsilon b.
 \end{cases}
\end{equation}


When the parameter $\epsilon$ goes to zero, assuming that $\kappa(\epsilon) = \frac{1}{\alpha}\epsilon + o(\epsilon) $ as $\epsilon \rightarrow 0$, we obtain the modified Kimura equation: 
\begin{equation}\label{Eq_Final_LMA_epsilon_0}
 \begin{cases}
  &  \pp_t \rho = \pp_{xx}^2 ( x(1 - x) \rho), \quad x \in (\delta, 1 - \delta), t >0 \\ 
  & \pp_x (x (1 - x) \rho) |_{x = \delta} = \alpha \rho (\delta, t) \\ 
  &  \pp_x (x (1 - x) \rho)  |_{x = 1 - \delta} =  - \alpha \rho(1 - \delta, t) \\
 \end{cases}
\end{equation}
along with
\begin{equation}\label{boundary_procedures}
 a'(t) =  \alpha\rho(\delta, t), \quad b'(t) = \alpha\rho(1 - \delta, t)
\end{equation}

\begin{remark}
Unlike the case of $\epsilon > 0$. The system (\ref{Eq_Final_LMA_epsilon_0}) is a closed system with a Robin boundary condition. Although the energy-dissipation law (\ref{ED_RK}) no longer holds with $\epsilon = 0$, the system can be interpreted as weighted $L^2$-type gradient flow 
\begin{equation}
 \frac{\dd}{\dd t} \left( \int_{\delta}^{1 - \delta} |\pp_x ( x(1 - x) \rho) |^2\dd x + \alpha \delta(1 - \delta) (|\rho(\delta, t)|^2 + |\rho(1 - \delta, t)|^2)  \right) = - \int_{\delta}^{1-\delta} x(1 - x) |\rho_t|^2 \dd x
\end{equation}
The variational structure gives another natural discretization of the modified Kimura equation in Eulerian coordinates. 
\end{remark}

 One of the important properties of the classical Kimura equation is the conservation of fixation probability, which corresponds to the conservation of the first moment in the pure drift case. For the modified system, we define the first moment as 
\begin{equation}\label{first_moment}
    \mathcal{M}(t) = \int^{\delta}_{0} x \frac{a(t)}{\delta} dx + \int_{\delta}^{1 - \delta} x\rho(x,t) dx + \int_{1 - \delta}^{1}x\frac{b(t)}{\delta}dx .
\end{equation}
The definition is based on the assumption that the probability density on $(0, \delta)$ and $(1 - \delta, 1)$ are defined by $\frac{a(t)}{\delta}$ and $\frac{b(t)}{\delta}$, respectively.
It is straightforward to show the following result for the defined first moment:
\begin{prop}\label{first_moment_deri}
    The derivative of the first moment $\mathcal{M}(t)$ defined in \eqref{first_moment} satisfies the following equation:  
    \begin{equation}
        \frac{d}{dt} \mathcal{M}(t) =  \left(\frac{\alpha}{2} - (1 - \delta)\right)\delta( ( \rho(1 - \delta, t) - \rho(\delta,t) ).     \end{equation} 
\end{prop}
\begin{proof}
From \eqref{Eq_Final_LMA_epsilon_0} and \eqref{boundary_procedures}, we have 
\begin{equation}\label{first_moment_derivative}
\begin{split}
& \frac{d}{dt} \left(\int^{\delta}_{0} x \frac{a(t)}{\delta} dx + \int_{\delta}^{1 - \delta} x\rho(x,t) dx + \int_{1 - \delta}^{1}x\frac{b(t)}{\delta}dx \right) \\ 
& = \frac{\delta}{2}a'(t) + \int_{\delta}^{1 - \delta}x\rho_{t}(x,t) dx + \left(1 - \frac{\delta}{2}\right)b'(t) \\ 
& = \frac{\alpha \delta}{2}\rho(\delta,t) + \int_{\delta}^{1 - \delta} x\partial_{xx}^{2}\left(x(1 - x)\rho(x,t)\right) dx + \left(1 - \frac{\delta}{2}\right)\alpha\rho(1 - \delta, t). 
\end{split}
\end{equation}
Using integration by parts and \eqref{Eq_Final_LMA_epsilon_0}, we obtain 
\begin{equation*}
\begin{split}
    \int_{\delta}^{1 - \delta} x\partial_{xx}^{2}\left(x(1 - x)\rho(x,t)\right) dx & = x \partial_{x}(x(1 - x)\rho)|^{1 - \delta}_{\delta} - \int_{\delta}^{1- \delta}\partial_{x}(x(1-x)\rho)dx \\
    & = -\alpha(1- \delta)\rho(1-\delta,t) - \alpha \delta\rho(\delta, t) - \delta(1 - \delta)\rho(1-\delta, t) + \delta(1 - \delta)\rho(\delta,t).
\end{split}
\end{equation*}
By substituting the above equations into \eqref{first_moment_derivative}, we finally have 
\begin{equation*}
    \frac{d}{dt}\mathcal{M}(t) =  \left(\frac{\alpha}{2} - (1 - \delta)\right)\delta \left[\rho(1 - \delta, t) - \rho(\delta, t) \right].
\end{equation*}
\end{proof}

From proposition \eqref{first_moment_deri}, it can be seen that the change in the first moment $\mathcal{M}(t)$ over time is $O(\delta)$ and the first moment is conserved if we take $\alpha = 2(1 - \delta)$. We'll take $\alpha = 2 (1 - \delta)$ for the remainder of this paper, unless stated otherwise.

\begin{remark}
In the previous paper \cite{liu2023continuum},  $a(t)$ and $b(t)$ are defined as the probability at $x = 0$ and $x = 1$. Under this viewpoint, to guarantee the conservation of the first moment, defined by $\int_{\delta}^{1-\delta} x \rho(x) \dd x + b(t)$, we need to have  $\alpha = (1 - \delta)$.
\end{remark}

 Next, we analyze the evolution of energy of the entire system, which we define as
\begin{equation}\label{Whole_domain_energy}
    \mathcal{E}(t) = \int_{0}^{\delta} \frac{a(t)}{\delta} \log\left(\frac{a(t)}{\delta}x(1 - x)\right) dx + \int_{\delta}^{1 - \delta} \rho(x,t)\log\left(\rho(x,t)x(1 - x)\right)dx + \int_{1 - \delta}^{1}\frac{b(t)}{\delta} \log\left(\frac{b(t)}{\delta}x(1 - x)\right)dx. 
\end{equation}
\begin{prop}\label{Whole_domain_energy_derivative}
    The derivative of the energy $\mathcal{E}(t)$, as defined in \eqref{Whole_domain_energy}, satisfies the following equation:
    \begin{equation}
    \begin{split}
        \frac{d}{dt} \mathcal{E}(t) & = - \int_{\delta}^{1 - \delta} \frac{|\partial_x\left(x(1 - x)\rho\right)|^{2}}{x(1 - x)\rho} dx + \alpha \rho(\delta, t)\left[\log(a(t))  - \log(\rho(\delta,t)) - (1 - \delta) - \frac{\log(1 - \delta)}{\delta}- \log(\delta)\right] \\
    & + \alpha \rho(1-\delta,t)\left[\log(b(t)) - \log(\rho(1 - \delta,t)) - (1 - \delta)  - \frac{\log(1- \delta)}{\delta} - \log(\delta)\right].
    \end{split}
    \end{equation}
\end{prop}
\begin{proof}
For the first term of \eqref{Whole_domain_energy}, we have 
\begin{equation*}
\begin{split}  
    \frac{d}{dt} \int_{0}^{\delta} \frac{a(t)}{\delta}\log\left(\frac{a(t)}{\delta}x(1- x)\right) dx & = \frac{d}{dt}\left[\int_{0}^{\delta} \frac{a(t)}{\delta}\log\left(\frac{a(t)}{\delta}\right)dx + \int_{0}^{\delta} \frac{a(t)}{\delta}\log(x(1-x)) dx \right] \\
    & = \frac{d}{dt}\left[a(t)\log\left(\frac{a(t)}{\delta}\right) + \frac{a(t)}{\delta}\left(\delta\log\delta - (1 - \delta)\log(1 - \delta)\right)\right] \\
    & = a'(t)\left[\log\left(\frac{a(t)}{\delta}\right) + \delta \right] + a'(t)\left(\log\delta - \frac{1 -\delta}{\delta}\log(1 - \delta) \right). 
\end{split}
\end{equation*} 
Applying \eqref{boundary_procedures} to the above equations, we obtain 
\begin{equation}
    \frac{d}{dt} \int_{0}^{\delta} \frac{a(t)}{\delta}\log\left(\frac{a(t)}{\delta}x(1- x)\right) dx = \alpha \rho(\delta,t)\left[\log(a(t)) - \frac{1 - \delta}{\delta}\log(1 - \delta) + \delta \right].
\end{equation}
Similarly, for $b(t)$, we have 
\begin{equation}
    \frac{d}{dt}\int_{1 - \delta}^{1}\frac{b(t)}{\delta} \log\left(\frac{b(t)}{\delta}x(1 - x)\right)dx = \alpha \rho(1 - \delta, t)\left[\log(b(t)) - \frac{1 - \delta}{\delta}\log(1 - \delta) + \delta \right].
\end{equation}
Now, for the bulk part of the free energy,  
by a direct calculation, we have
\begin{equation}\label{bulk_free_energy}
    \begin{split}
        &  \frac{d}{dt} \int_{\delta}^{1 - \delta} \rho \log\left(x(1 - x)\rho \right) dx  = - \int_{\delta}^{1 - \delta} \frac{|\partial_x\left(x(1 - x)\rho\right)|^{2}}{x(1 - x)\rho} dx \\ & \quad -  \alpha\rho(1- \delta,t) \left[\log\left(\delta(1 - \delta)\right) + \log \rho(1 - \delta,t) 
           + 1 \right] - \alpha \rho(\delta,t)\left[\log(\delta(1 - \delta)) + \log\rho(\delta,t) + 1\right]\ ,
    \end{split}
    \end{equation}
Hence, combining the three terms, we finally get the desired result. 
\end{proof}
\begin{remark}
    The proposition \eqref{Whole_domain_energy_derivative} does not guarantee energy dissipation at all times because the contribution of the boundary terms may be positive in the derivative of the energy. However, the boundary terms get smaller and approach zero as the density diminishes over time.
\end{remark} 
\section{A Structure-preserving discretization to the modified Kimura equation}

In this section, we propose a structure-preserving scheme for the modified Kimura equation \eqref{Eq_Final_LMA_epsilon_0} along with the boundary dynamics \eqref{boundary_procedures}. 

As mentioned above, the energy-dissipation law \eqref{ED_RK} no longer holds for $\epsilon = 0$. Instead, the system \eqref{Eq_Final_LMA_epsilon_0} satisfies the energy identity (\ref{bulk_free_energy}), where the Robin-type boundary condition of $\rho(x, t)$ may contribute to an increase in the defined free energy. Additionally, $\rho(x, t)$ ($x \in [\delta, 1 - \delta]$) is no longer a conserved quantity. Consequently, we cannot directly apply Lagrangian-type methods commonly used for diffusion equations \cite{duan2019numerical, carrillo2010numerical, liu2020lagrangian} to the modified system.

To overcome these difficulties, we propose a Lagrangian-Eulerian hybrid operator splitting scheme for the equations \eqref{Eq_Final_LMA_epsilon_0} and \eqref{boundary_procedures}.
The method is described below. 
\begin{itemize}
\item Step 1: Given $\rho^n$, solve the equations on $(\delta, 1 - \delta)$ with the no-flux boundary condition
\begin{equation}\label{step1}
 \begin{cases}
  &  \pp_t \rho = - \pp_{x} (\rho u) , \quad x \in (\delta, 1 - \delta), t \in [t^{n}, t^{n} + \Delta t] \\ 
  &  \rho u = - \pp_{x} \left(x (1 - x) \rho \right), \quad x \in (\delta, 1 - \delta) \\ 
  & \pp_x (x (1 - x) \rho) |_{x = \delta} = 0 \\
  &  \pp_x (x (1 - x) \rho)  |_{x = 1 - \delta} =  0 \\
 \end{cases}
 \end{equation}
with the initial condition $\rho(x, t^n) = \rho^n(x)$ to obtain $\tilde{\rho}^{n+1}$. Note that since the equation is a diffusion equation defined on $(\delta, 1 - \delta)$ with no-flux boundary condition, Lagrangian type methods \cite{liu2020lagrangian, duan2019numerical} can be applied.


\item Step 2: Given $\tilde{\rho}^{n+1}$, $a^{n}$, and $b^{n}$, solve 
the boundary dynamics

\begin{equation}\label{boundary_ODE}
 a'(t) =  \alpha \tilde{\rho}^{n+1}(\delta), \quad b'(t) = \alpha \tilde{\rho}^{n+1}(1 - \delta)
\end{equation}

for $t \in (t^n + t^n + \Delta t)$ with the initial condition
\[  a(t^n) = a^n, \quad  b(t^n) = b^n,  \]
to get $a^{n+1}$ and $b^{n+1}$, and update the density $\tilde{\rho}^{n+1}$ to $\rho^{n+1}$ by updating the density at the boundary.
\end{itemize}
\subsection{Step 1: A Lagrangian scheme for the interior dynamics}

Since the equation (\ref{step1}) is a diffusion with non-flux boundary condition, we can develop a Lagrangian scheme to solve it. At each time step, given $\rho^n$, 
the system (\ref{step1}) satisfies the energy-dissipation law
\begin{equation}\label{Lagrangian_energy_tn}
    \dfrac{d}{dt} \int_{\delta}^{1 - \delta} \rho^n (X) \left(\log(x(1 - x)) + \log\left(\dfrac{\rho^n(X)}{\det F(X, t)}\right)\right)\ dX = -\int_{\delta}^{1 - \delta} \dfrac{\rho^n(X)}{x(1-x)}|x_{t}|^{2} \ dX.
\end{equation} 
Here, $\rho^n$ is the numerical solution at $t^n$ and $x$ denotes the flow map $x(X, t)$ in $(t^n, t^{n+1})$. 

The idea of Lagrange method is to discretize the flow map $x(X, t)$ directly. In the current study, we apply a finite difference method to discretize the flow map. To derive the scheme, we apply a discrete variational approach \cite{liu2020lagrangian}, which first discretizes the energy-dissipation law \eqref{Lagrangian_energy_tn} and then takes variation to obtain a semi-discrete scheme. The approach is different from the traditional equation-based discretization, and has advantages in preserving the variational structure at the semi-discrete level \cite{liu2020lagrangian}.

Let $\{ \delta = x_{0}^n < x_{1}^n < \cdots < x_{N-1}^n < x_{N}^n = 1 - \delta \}$ denote the Lagrangian reference points at time $t^n$, and define the grid spacing as $h_i^n = x_i^n - x_{i-1}^n$ for $i = 1, \ldots, N$. Since we are only concerned with the discretization over the time interval $(t^n, t^{n+1})$, we simplify the notation by letting $X_i := x_i^n$ represent the Lagrangian reference points and $h_i := h_i^n$ the corresponding grid spacings. The choice of Lagrangian reference points at each time step will be discussed later.

Let $x_i(t)$ denote the trajectory of the $i$-th grid point over $(t^n, t^{n+1})$, satisfying the initial condition $x_i(t^n) = X_i$. The flow map $x(X, t)$ can then be approximated at the grid points $\{X_i\}_{i=0}^N$ by
\[
x_h(X_i, t) = x_i(t), \quad i = 0, 1, \ldots, N.
\]
$x_h(X, t)$ can be viewed as a grid function on
\[ \mathcal{I}^{n} = \{X_{i}, \ i = 0,\cdots, N\} \]
Accordingly, the deformation tensor (\ref{def_F}) can be approximated at the half-grid points $X_{i+1/2} := X_i + h_{i+1}/2$ by
\[
\det F_h (X_{i+1/2}, t) = \frac{x_{i+1}(t) - x_i(t)}{X_{i+1} - X_i}.
\]
using the finite difference approximation, which is a grid function on
\[ \mathcal{H}^{n} = \{X_{i+1/2} = (X_{i} + X_{i+1})/2, \ i = 0, \cdots, N-1\} \]
Clearly, the trajectories $\{x_i(t)\}_{i=0}^N$ must belong to the admissible set
\[
\mathcal{Q} = \left\{ x = (x_0, x_1, \ldots, x_N) \ \middle| \ \delta = x_0 < x_1 < \cdots < x_{N-1} < x_N = 1 - \delta \right\}.
\]
The boundary of $\mathcal{Q}$ is defined as
\[
\partial \mathcal{Q} = \left\{ x = (x_0, x_1, \ldots, x_N) \ \middle| \ \delta = x_0 \leq x_1 \leq \cdots \leq x_N = 1 - \delta, \ \text{and} \ x_i = x_{i-1} \ \text{for some} \ 1 \leq i \leq N \right\}.
\]
We can view $\{ x_i(t) \}_{i=1}^N$ as Lagrangian particles \cite{liu2020lagrangian, westdickenberg2010variational}.

The goal is to derive the ODE of $x_i(t)$ from the energy-dissipation law \eqref{Lagrangian_energy_tn}.
To this end, we first discretize the energy-dissipation law by approximating the integral in \eqref{Lagrangian_energy_tn} on each subinterval $(X_i, X_{i+1})$. Recall the kinemtics of the density $\rho$ \eqref{density_Lagrangian}, we can approximate the density $\rho(x(X, t), t)$ at the half-grid points $X_{i+1/2}$ by
\begin{equation}
\rho(x(X_{i+1/2}, t), t)  = \rho^n_{i+1/2} / \det F_h (X_{i+1}, t)\ ,  t \in (t^n, t^{n+1}) \ ,
\end{equation}
which can be viewed as a grid function on $\mathcal{H}^n$. Here, $\rho^n_{i+1/2}$ can be view as $\rho^n (X_{i+1/2})$ or cell average of $\rho^n$ on the interval $(X_i, X_{i+1})$.



Given the grid points $\{ x_i(t) \}{i=1}^N$ in $\mathcal{Q}$, and noting that the density $\rho$ and the deformation tensor $F$ are approximated by grid functions on $\mathcal{H}^n$, while the flow map $x(X, t)$ is approximated by a grid function on $\mathcal{I}^n$, we approximate the bulk free energy $\mathcal{F}^{n}$ as follows: \begin{equation}
    \label{eqn:discretized_energy}
    \begin{split}
         \mathcal{F}_{h}(\{x_{i}(t)\}_{i=0}^{N}) 
        & = \sum_{i=0}^{N-1} \rho^{n}_{i+1/2}  \left( \frac{\log(x_{i} (1 - x_{i})) + \log(x_{i+1}(1 - x_{i+1})}{2} +  \log\left(\frac{\rho^{n}_{i+1/2}}{\frac{x_{i+1} - x_{i}}{h^{n}_{i}}}\right) \right) h^{n}_{i}. \\
    \end{split} 
    \end{equation} 

\begin{remark} 
    The approximation in (\ref{eqn:discretized_energy}) is obtained by first replacing $\rho^n$ with its piecewise constant approximation: \begin{equation} \label{piecewise_constant} \rho^n(X) = \sum_{i=0}^{N-1} \rho^{n}_{i+1/2} \mathbf{1}_{(X_i, X_{i+1})}(X), \end{equation} in the continuous free energy functional, and then applying the trapezoidal rule to approximate the integral $\int_{X_i}^{X_{i+1}} x(X, t) \ln x(X, t), \mathrm{d}X$. Here, $\mathbf{1}_{(X_i, X{i+1})}(X)$ denotes the characteristic function of the interval $(X_i, X_{i+1})$. 
\end{remark}

Similarly, for the dissipation term, we adopt the piecewise constant approximation for the density and apply the trapezoidal rule to approximate the corresponding integral. This leads to the following discretized dissipation functional:
\begin{equation}
    \begin{split}
        \label{eqn:trapezoidal_dissipation}
        \mathcal{D}_{h}(\{(x_{i})_{t}\}_{i=0}^{N}) 
        & = \frac{1}{2}\sum_{i=0}^{N-1} \frac{1}{2} \rho^{n}_{i+1/2} \left[\frac{|(x_{i})_{t}|^{2}}{x_{i}(1 - x_{i})} + \frac{|(x_{i+1})_{t}|^{2} }{x_{i+1} (1 - x_{i+1})}\right] h^{n}_{i}  
    \end{split}
\end{equation}

Based on these approximations, we obtain a discrete energy-dissipation law in terms of particles $\{ x_i(t) \}_{i=1}^N$. This discrete variational structure then allows us to apply the Least Action Principle (LAP) and the Maximum Dissipation Principle (MDP) to derive the governing equations for $x_i(t)$. By taking the variation of the discrete action functional with respect to $x_{i}$, we get 
\begin{equation}
    \label{eqn:variation_discrete_energy}
    \frac{\delta \mathcal{A}_{h}}{\delta x_{i}} = - \frac{1 - 2x_{i}}{2x_{i}(1 - x_{i})} \left[ \rho_{i-1/2}^{n} h^{n}_{i-1} + \rho_{i+1/2}^{n} h^{n}_{i} \right] - \frac{\rho_{i-1/2}^{n}h^{n}_{i-1}}{x_{i} - x_{i-1}} + \frac{\rho_{i+1/2}^{n}h^{n}_{i}}{x_{i+1} - x_{i}}, \ 1 \leq i \leq N-1.
\end{equation} 
On the other hand, taking variation of $\mathcal{D}_{h}$ with respect to $(x_{i})_{t}$  will give us 
\begin{equation}
    \label{eqn:variation_discrete_dissipation}
    \frac{\delta \mathcal{D}_{h}}{\delta (x_{i})_{t}} = \frac{\rho_{i-1/2}^{n}h^{n}_{i-1} + \rho_{i+1/2}^{n}h^{n}_{i}}{2x_{i}(1 - x_{i})}(x_{i})_{t}, \ 1 \leq i \leq N-1.
\end{equation}
Finally, by applying the force balance we obtain the semi-discrete equations
\begin{equation}
    \label{eqn:discrete_force_balance_ODE_0}
    \begin{split}
      & \frac{\rho_{i-1/2}^{n}h^{n}_{i-1} + \rho_{i+1/2}^{n}h^{n}_{i}}{2x_{i}(1 - x_{i})} (x_{i})_{t} = -\frac{\rho_{i-1/2}^{n}h^{n}_{i-1} + \rho_{i+1/2}^{n}h^{n}_{i}}{2x_{i}(1 - x_{i})}(1 - 2 x_{i}) + \frac{\rho_{i-1/2}^{n}h^{n}_{i-1}}{x_{i} - x_{i-1}} - \frac{\rho_{i+1/2}^{n}h^{n}_{i}}{x_{i+1} - x_{i}}  , \  1 \leq i \leq N-1.
\end{split}
\end{equation}

\begin{remark}
    The equation~(\ref{eqn:discrete_force_balance_ODE_0}) can be interpreted as a finite-difference approximation to the equation of flow map \( x(X, t) \):
    \begin{equation}
        \frac{1}{x (1 - x)} \rho^n(X) x_t = -  \partial_X \left( \frac{\rho^n (X)}{\det F} \right) - \frac{1}{x(1 - x)} \rho^n(X) (1 - 2 x(X, t)), \quad t \in (t^n, t^{n+1})\ ,
    \end{equation}
    which can be obtained by writing the continuous velocity equation~(\ref{velocity_1}) in Lagrangian coordinates, and cancel the additional factor of \( \det F \) by using the identity \( F = \det F \) for the one-dimensional deformation gradient $F$. In contrast to \cite{duan2019numerical}, we define the Lagrangian reference density  $\rho^n$ as a grid function on $\mathcal{H}^n$, rather than on $\mathcal{I}^n$ and use the approximation
    \[
     \rho^n(X_i) \approx \frac{1}{2} (\rho_{i-1/2}^n h_{i-1}^n + \rho_{i+1/2}^n h_i^n),
    \]
\end{remark}

There are several ways to obtain the fully discretized scheme by introducing a suitable temporal discretization to \eqref{eqn:discrete_force_balance_ODE_0} numerically. Since \eqref{eqn:discrete_force_balance_ODE_0} is a gradient flow with nonlinear mobility, a standard approach is to use an implicit Euler scheme to \eqref{eqn:discrete_force_balance_ODE_0}, but keeping the mobility term on the left-hand side explicit, which leads to  
\begin{equation}\label{eqn:ODE_0}
\begin{split}
    & \frac{\rho_{i-1/2}^{n}h^{n}_{i-1} + \rho_{i+1/2}^{n}h^{n}_{i}}{2X_{i}(1 - X_{i})}\frac{x_{i}^{n+1} - X_{i}}{\tau}  = -  \frac{\rho_{i-1/2}^{n}h^{n}_{i-1} + \rho_{i+1/2}^{n}h^{n}_{i}}{2x^{n+1}_{i}(1 - x^{n+1}_{i})}(1 - 2x_{i}^{n+1}) + \frac{\rho_{i-1/2}^{n}h^{n}_{i-1}}{x_{i}^{n+1} - x_{i-1}^{n+1}} - \frac{\rho_{i+1/2}^{n}h^{n}_{i}}{x_{i+1}^{n+1} - x_{i}^{n+1}},
\end{split}
\end{equation} 
where $1 \leq i \leq N-1$ and  $x_i^n = X_i$ is used.  The implicit Eulerian discretization can be reformulated as the following optimization problem: 
\begin{equation} \label{optimization}
    \begin{split}
     \{ x_i^{n+1} \}_{i=0}^N = \argmin_{ \{ y_i \}_{i=0}^N \in \mathcal{Q}}   J( \{ y_i \}_{i=0}^N), \quad    J( \{ y_{i} \}_{i=0}^N) := \sum_{i=1}^{N-1}\frac{\rho_{i-1/2}^{n}h^{n}_{i-1} + \rho_{i+1/2}^{n}h^{n}_{i}}{2X_{i}(1 - X_{i})} \frac{(y_i - X_i)^2}{2 \tau} +  \mathcal{F}( \{ \{ y_i \}_{i=0}^N \}) . 
\end{split}
\end{equation}
Since the first term in $( \{ y \}_{i=0}^N)$  is always positive, this step always decrease the energy $\mathcal{F}(\{ x_i \}_{i=1}^N)$, i.e.,
\begin{equation}
\mathcal{F}(\{ x_i^{n+1} \}_{i=1}^N) \leq \mathcal{F}(\{ x_i^{n} \}_{i=1}^N).
\end{equation}
Theoretically, we can show that $J( \{ y \}_{i=0}^N)$ is a convex function in the admissible set $\mathcal{Q}$, provided $\tau$ is sufficiently small. More precisely, we have the following proposition:
\begin{prop} 
    Let
    \begin{equation*}
    \mathcal{Q} = \{ {\bm x} = (x_0, x_1, \ldots, x_N)^{\rm T} | \   \delta = x_{0} < x_{1} < \cdots < x_{N-1} < x_{N} = 1 - \delta \},
    \end{equation*}
    be the admissible set, and ${\bm X} = (X_0, X_1, \ldots X_N)^{\rm T} \in \mathcal{Q}$, there exists a small time step $\tau$ of the same order as $\delta^{2}$ such that $J( {\bm y})$, defined in (\ref{optimization}), is convex on $\mathcal{Q}$.
\end{prop}
\begin{proof}
    Taking the second derivatives of $J({\bm y})$, we obtain
    \begin{equation*}
    \begin{cases}
        & \frac{\partial^{2}J}{\partial y_{i-1}\partial y_{i}} = - \frac{\rho_{i-1/2}^{n}h_{i-1}^{n}}{(y_{i} - y_{i-1})^{2}}, \\ \\
        &\frac{\partial^{2}J}{\partial y_{i}^{2}} = \frac{1}{\tau}\frac{\rho_{i-1/2}^{n}h_{i-1}^{n} + \rho_{i+1/2}^{n}h_{i}^{n}}{X_{i}(1 -X_{i})} + \frac{\rho_{i-1/2}^{n}h_{i-1}^{n} + \rho_{i+1/2}^{n}h_{i}^{n}}{2}\frac{2y_{i}(1-y_{i}) - 1}{(y_{i}(1- y_{i})^{2}} + \frac{\rho_{i-1/2}^{n}h_{i-1}^{n}}{(y_{i} - y_{i-1})^{2}} + \frac{\rho_{i+1/2}^{n}h_{i}^{n}}{(y_{i+1} - y_{i})^{2}},  \\ \\ 
        & \frac{\partial^{2}J}{\partial y_{i+1} \partial y_{i}} = - \frac{\rho_{i+1/2}^{n}h_{i}^{n}}{(y_{i+1} - y_{i})^{2}}, i = 1, \cdots, N-1.\end{cases}
    \end{equation*}
    Hence, the Hessian matrix of $J$ is diagonally dominant if 
    \begin{equation*}
        \frac{1}{\tau}\frac{\rho_{i-1/2}^{n}h_{i-1}^{n} + \rho_{i+1/2}^{n}h_{i}^{n}}{X_{i}(1 -X_{i})} + \frac{\rho_{i-1/2}^{n}h_{i-1}^{n} + \rho_{i+1/2}^{n}h_{i}^{n}}{2}\frac{2y_{i}(1-y_{i}) - 1}{(y_{i}(1- y_{i})^{2}} \geq 0, \quad \forall i = 1, \cdots, N-1.
    \end{equation*}
    After some algebraic manipulation, the above inequality is equivalent to: 
    \begin{equation}\label{inequality}
        \frac{1}{\tau}\frac{2(y_{i}(1- y_{i}))^{2}}{X_{i}(1 - X_{i})} + 2y_{i}(1 - y_{i}) - 1\geq 0, \quad i = 1, \cdots, N-1. 
    \end{equation}
    Note that we have $\delta(1 - \delta) \leq x(1-x) \leq \frac{1}{4}, \forall x\in (\delta, 1 - \delta)$. By substituting the uniform bound into the inequality, we have  
    \begin{equation*}
        \frac{1}{\tau}\frac{2(y_{i}(1- y_{i}))^{2}}{X_{i}(1 - X_{i})} + 2y_{i}(1 - y_{i}) - 1 \geq \frac{8(\delta(1- \delta))^{2}}{\tau} + 2\delta(1 - \delta) - 1, \quad i = 1, \cdots, N-1.
    \end{equation*}
    Therefore, \eqref{inequality} holds if we take $\tau$ to be the same order of $\delta^{2}$ such as 
    \begin{equation*}
        \tau \leq \dfrac{8(\delta(1- \delta))^{2}}{1 - 2\delta(1 - \delta) }= O(\delta^{2}). 
    \end{equation*}
\end{proof}

\begin{remark}
    In \cite{duan2019numerical, carrillo2022optimal}, the authors adopt a convex splitting scheme to solve a similar equation for $x_i$ when $\delta = 0$. It is important to note that $J(y)$ is not bounded from below if $\delta = 0$ due to the presence of the $\ln x(1 - x)$ terms. Hence, a convex splitting scheme is necessary in this case. For $\delta > 0$, we can use a fully implicit discretization, and the convexity of $J(y)$ can be proven if $\tau$ is sufficiently small. The fully implicit scheme may offer certain advantages over convex splitting schemes. However, when $\delta$ is too small, a convex splitting approach may still be required.
\end{remark}

Although the convexity of the optimization problem (\ref{optimization}) is guaranteed, the numerical scheme \eqref{eqn:ODE_0} may not be stable when $\delta$ is very small even with small temporal step size. This is because the term $\frac{1}{X(1-X)}$ in \eqref{eqn:ODE_0} can become large when $X$ is close to 0 or 1. As a result, it is difficult to choose a suitable step size for gradient-based algorithm in solving the optimization problem (\ref{optimization}) such that $y$ stay in $\mathcal{Q}$.
To address this drawback of the standard semi-implicit method,  we propose an alternative approach by multiplying both sides of the discretized force balance equation (\ref{eqn:discrete_force_balance_ODE_0}) by $x_{i}(1 - x_{i})$ first, which leads 
\begin{equation}
    \label{eqn:discrete_force_balance_ODE}
     \frac{\rho_{i-1/2}^{n}h^{n}_{i-1} + \rho_{i+1/2}^{n}h^{n}_{i}}{2}  (x_{i})_{t} =  (x_i(1 - x_i))\left( \frac{\rho_{i-1/2}^{n}h^{n}_{i-1}}{x_{i} - x_{i-1}} - \frac{\rho_{i+1/2}^{n}h^{n}_{i}}{x_{i+1} - x_{i}} \right)  - \frac{\rho_{i-1/2}^{n}h^{n}_{i-1} + \rho_{i+1/2}^{n}h^{n}_{i} }{2}   (1 - 2 x_i)
\end{equation} 
for $1 \leq i \leq N-1$. The equation (\ref{eqn:discrete_force_balance_ODE}) can be interpreted as a finite-difference approximation of the velocity equation (\ref{velocity}) in the Lagrangian coordinates. 
By applying the same implicit Euler discretization to \eqref{eqn:discrete_force_balance_ODE}, we obtain a {\it new} scheme, which can be written as
\begin{equation}\label{eqn:ODE}
\begin{split}
    \frac{\rho_{i-1/2}^{n}h^{n}_{i-1} + \rho_{i+1/2}^{n}h^{n}_{i}}{2} \left(\frac{x_{i}^{n+1} - X_{i}}{\tau} + 1 - 2x_{i}^{n+1}\right) = \left[\frac{\rho_{i-1/2}^{n}h^{n}_{i-1}}{x_{i}^{n+1} - x_{i-1}^{n+1}} - \frac{\rho_{i+1/2}^{n}h^{n}_{i}}{x_{i+1}^{n+1} - x_{i}^{n+1}}\right] x_{i}^{n+1}(1 - x_{i}^{n+1})
\end{split}
\end{equation} 
for $1 \leq i \leq N-1$. The scheme (\ref{eqn:ODE}) can be obtained from (\ref{eqn:discrete_force_balance_ODE_0}) by treating $x_i(1 - x_i)$ in the mobility implicitly but keeping the other terms explicitly. 
Although it is might be difficult to reformulate the scheme (\ref{eqn:ODE}) is to an optimization problem like (\ref{optimization}), we can still apply the gradient decent with the Barzilai-Borwein (BB) method \cite{barzilai1988two}, which is indeed a fixed point iteration method. 
Strictly speaking, the full discretized scheme didn't maintain the original variational structure. However, numerical tests show that the new scheme is more stable than the previous one. Hence, we'll apply the second the scheme in all the numerical experiments below.

Next we discuss how to choose the Lagrange reference points at each time step. At the initial step, we select equidistant grid points to divide the computational domain $(\delta, 1 - \delta)$ into $N_0$ non-overlapping subintervals and initialize $x_i^0 = \delta + ih$ for $i = 0, \cdots, N_0$ as the initial grid points, where $h = \frac{1 - 2\delta}{N_0}$ is the subinterval length. At each time step, we first update the grid points using the Lagrangian scheme defined in equation (\ref{eqn:ODE}), yielding new positions $x_i^{n+1}$. We then apply a removal procedure to handle particles that move too close to the domain boundaries. Specifically, if a particle enters a small buffer region near either boundary, we merge it with the particle at $\delta$ or $1 - \delta$, respectively. The initial mass for each particle within the domain is defined as 
\begin{equation}\label{mass_initial}
     m_{i}^{0} = h \rho_{0}(X_{i+1/2}), \ 0 \leq i \leq N_0 - 1\ .
\end{equation}
Let
\begin{equation}\label{reindex_particles}
\begin{cases}
& i^{n+1}_{c} = \min \{\ i \ | \ x^{n+1}_{i} > \delta + \eta \}, \\
& i^{n+1}_{f} = \max\{\ i \ | \ x^{n+1}_{i} < 1 - \delta - \eta \}
\end{cases}
\end{equation}
where $\eta$ is the length of the buffer region. We then update the average density at each interior cell by  
\begin{equation}\label{update_interior_density}
    \tilde{\rho}^{n+1}_{i+1/2} = \frac{\rho^{n}_{i+1/2}}{(x^{n+1}_{i+1} - x^{n+1}_{i}) / h_i^n}, \quad i_{c}^{n+1}  \leq i \leq i^{n+1}_{f} - 1.
\end{equation} 
We remove the particles $x_1^{n+1}, \ldots, x_{i_{c}^{n+1} - 1}^{n+1}$, note that $x^{n+1}_{0} = \delta$ is fixed, and define the total mass and the average density in the interval $[x_0^{n+1}, x_{i_c^{n+1}}^{n+1}]$ by
\begin{equation}
 m_{0}^{n+1} = \sum_{i=1}^{i^{n+1}_{c} -1} m_i^0, \quad \tilde{\rho}^{n+1}_{1/2} = \frac{m_0^{n+1}}{x_{i_c^{n+1}}^{n+1} - x_0^{n+1}}
\end{equation}
A similar update rule is applied to the last cell. We then re-index all particles after the removal procedure. The total number of particles becomes $N_{n+1} = i_f^{n+1} - i_c^{n+1} + 3$. We omit the subindex $n+1$ without ambiguity.  After the first step, we define the boundary conditions  
\[
\tilde{\rho}^{n+1}_{l} = \tilde{\rho}^{n+1}_{1/2}, \quad \tilde{\rho}^{n+1}_{r} = \tilde{\rho}^{n}_{N - 1/2},
\]  
which will be used in the second step.

\subsection{Step 2: An Eulerian Scheme for the boundary dynamics}
In this step, we update the density value at the boundary. 
Given the boundary density values $\tilde{\rho}^{n+1}_{l}$ and $\tilde{\rho}^{n+1}_{r}$ obtained from step 1, we can update the boundary mass $a^{n+1}$ and $b^{n+1}$ as follows:
\begin{equation}\label{updating_boundary_mass}
\begin{cases}
& a^{n+1} = a^n + \tau  \alpha \tilde{\rho}^{n+1}_{l} \\
& b^{n+1} = b^{n} + \tau \alpha \tilde{\rho}^{n+1}_{r} .
\end{cases}
\end{equation}
One can view (\ref{updating_boundary_mass}) as an explicit Euler discretization for the ODE  \eqref{boundary_ODE}.

After obtaining $a^{n+1}$ and $b^{n+1}$, we update the density values at the boundary using the following formula:
\begin{equation}\label{updated_density}
\begin{cases}
    & \rho_{1/2}^{n+1} = \rho^{n+1}_{l} = \left(1 -  \dfrac{\tau \alpha}{x_{1}^{n+1} - \delta}\right) \tilde{\rho}^{n+1}_{l} \\      
    & \rho_{N - 1/2}^{n+1} \rho^{n+1}_{r} = \left(1 - \dfrac{\tau \alpha}{1 - \delta - x_{N -1}^{n+1}}\right) \tilde{\rho}^{n+1}_{r},
\end{cases}
\end{equation}
where $N = N_{n+1}$. The update rule \eqref{updated_density} ensures the mass conservation in the sense of 
\begin{equation}\label{total_mass}
    M^{n} = a(t_n) +  \int_{\delta}^{1 - \delta} \hat{\rho} (x, t_n) \dd x + b(t_n)\ .
\end{equation}
where
\begin{equation}\label{approximated_density}
    \hat{\rho}(x, t_{n+1}) = \sum_{i = 0}^{N - 1} \rho_{i+1/2}^{n+1} {\bf 1}_{(x_i, x_{i+1})} . 
   \end{equation}
is the piecewise constant approximation to $\rho(x, t^n)$.

\subsection{Numerical Methods for the Operator Splitting Scheme}

We summarize the above discussion by the following algorithm for obtaining the numerical density evolution of \eqref{Eq_Final_LMA}:  
\begin{algorithm}[H] 
\caption{}\label{algorithm}
\begin{enumerate}
    \item Initial setting. \\ 
    For $0 \leq i \leq N$, we are given the initial positions of the particles $x_{i}^{0} = X_{i} $, the initial density distribution function $\rho_{0}(X_{i+1/2})$, and choose the artificial parameters $\delta $ for our domain and $\eta$ as a threshold value to check if the particles move close to the boundary.
    \item Lagrangian scheme for the bulk domain. 
        \subitem 1.Update the positions of the particles by solving the problem \eqref{eqn:ODE}. 
        \subitem 2. Re-index the particles according to \eqref{reindex_particles}. 
        \subitem 3. Update the density at each cell by \eqref{update_interior_density}.
    \item Eulerian scheme for the boundary. 
        \subitem 1. Update the mass function $a(t)$ and $b(t)$ by \eqref{updating_boundary_mass}. 
        \subitem 2. Update the density at the first and last cell by \eqref{updated_density}.  
    
    \item Return the updated mass functions $a(t)$ and $b(t)$ and the discrete density function obtained from the previous steps. 
\end{enumerate}
\end{algorithm}

\section{Numerical Results}
In this section, we present some numerical results for the modified Kimura equation to demonstrate the efficiency, accuracy and structure-preserving property of the proposed scheme. We consider two initial density functions:
\begin{equation}
    \rho_{0}^1(x) = \frac{2 + 6 x + \frac{\pi}{2}\sin(2\pi( x - \delta)/(1 - 2\delta))}{5(1 - 2\delta)} \\ 
\end{equation}
and
\begin{equation}
\rho_{0}^1(x) = c_{1}\Phi_{1}(x) + c_{2}\Phi_{2}(x) + \frac{1 - c_{1} - c_{2}}{1 - 2\delta}, \quad x \in (\delta, 1 - \delta)\ .
\end{equation}
Here, $c_{1}$ and $c_{2}$ are set to $0.6$ and $0.2$, respectively. The functions $\Phi_{1}$ and $\Phi_{2}$ are truncated Gaussian distributions with standard deviation $\sigma = 0.1$, centered at $\mu_{1} = 0.2$ and $\mu_{2} = 0.7$, respectively. The term $1 - 2\delta$ in the denominator of the initial data ensures that the total integral equals $1$.
We choose the first density function to compare our numerical results with those of Duan \cite{duan2019numerical}. The second density function is selected to further demonstrate the property of our numerical scheme when applied to more complex initial conditions.

Figures~\ref{fig:density_f1_alpha_0.7_2(1-delta)} shows the evolution of densities (represented by circles) at time $t = 0.1$, $t = 0.2$ and $t = 1.5$ with $\alpha = 2(1-\delta)$ and $N = 150$. The reference Eulerian solution, represented by the blue line, is obtained using the Eulerian scheme proposed in \cite{liu2023continuum} with $h = \frac{1 - 2\delta}{10000}$. It can be seen that the Lagrangian solutions match well with the Eulerian solutions using fewer grid points. 

\begin{figure}[ht!] 
\centering 
    \begin{overpic}[width = 0.32\textwidth]{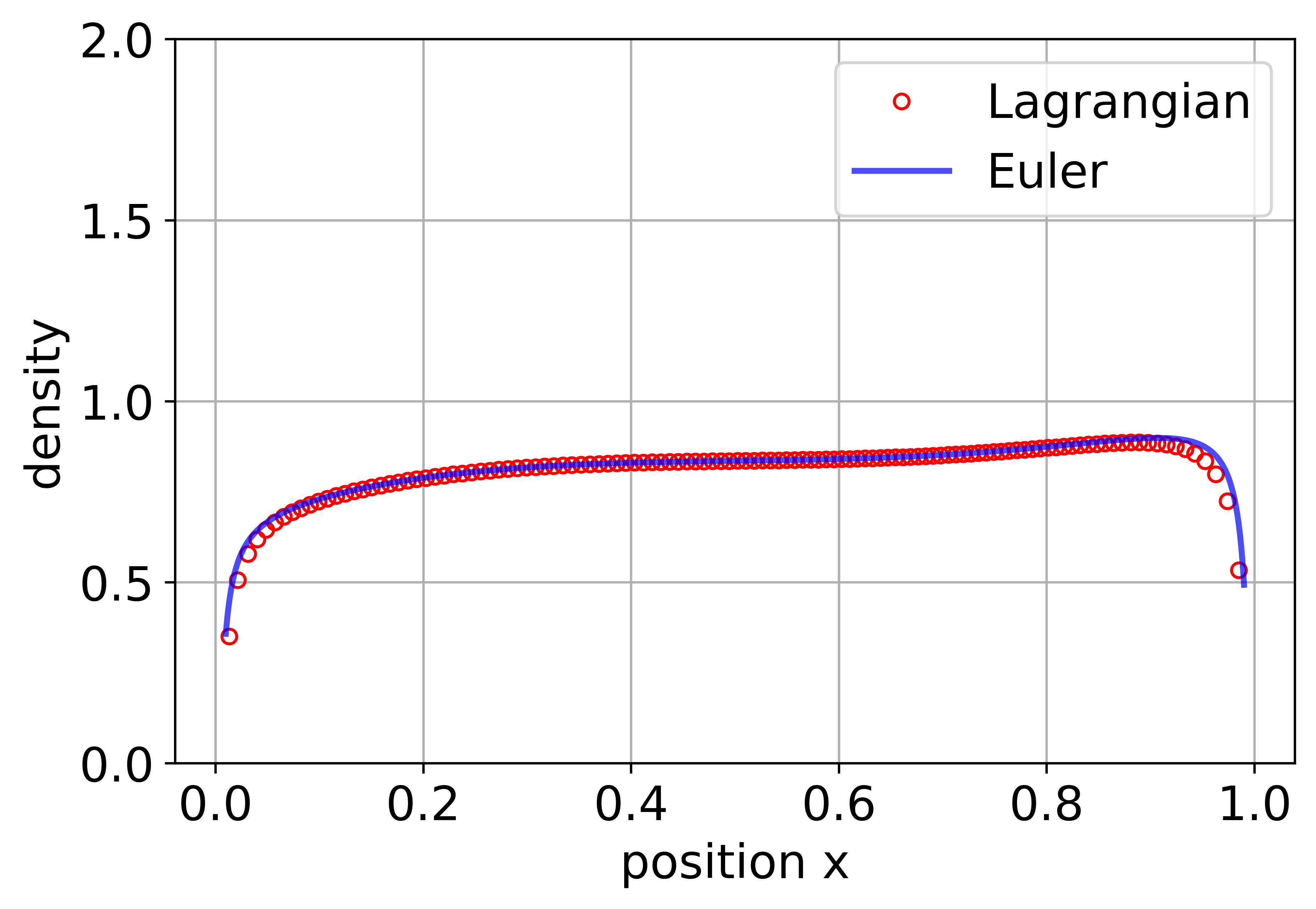}
        \put(-3,55){(a)}
    \end{overpic}
    \hfill
    \begin{overpic}[width=0.32\textwidth]{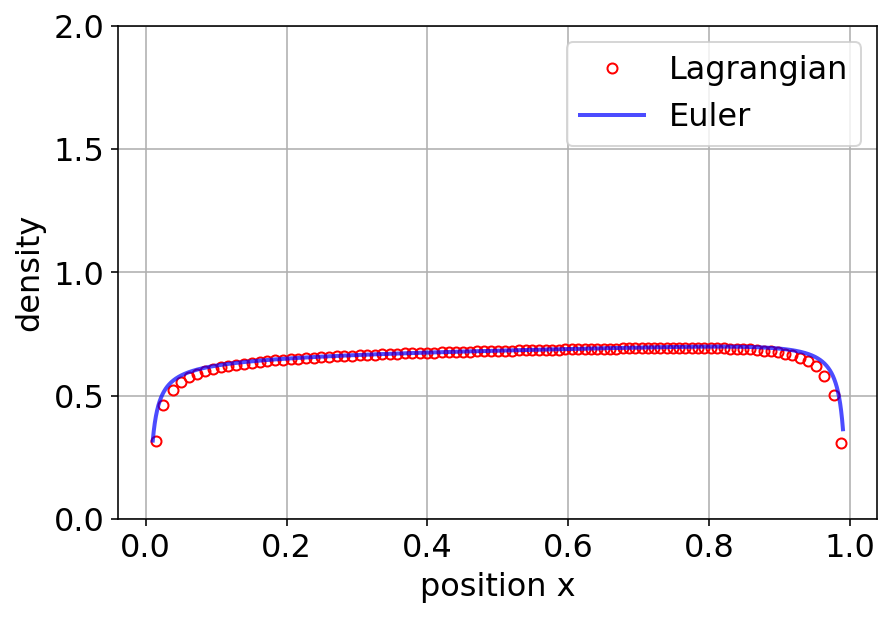}
        \put(-3,55){(b)}
    \end{overpic}
    \hfill
    \begin{overpic}[width = 0.32\textwidth]{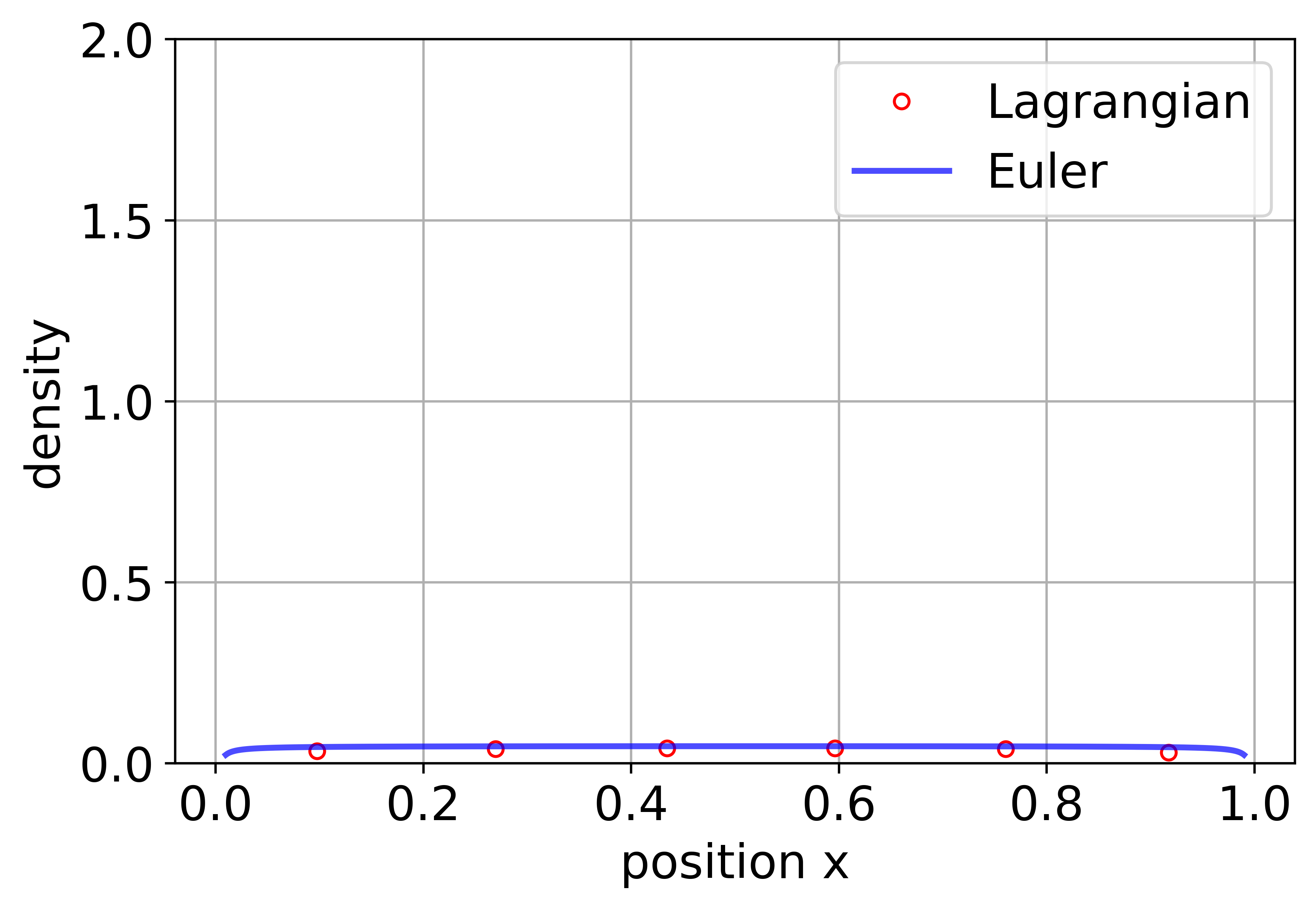}
        \put(-3,55){(c)}
    \end{overpic}

    \begin{overpic}[width = 0.32\textwidth]{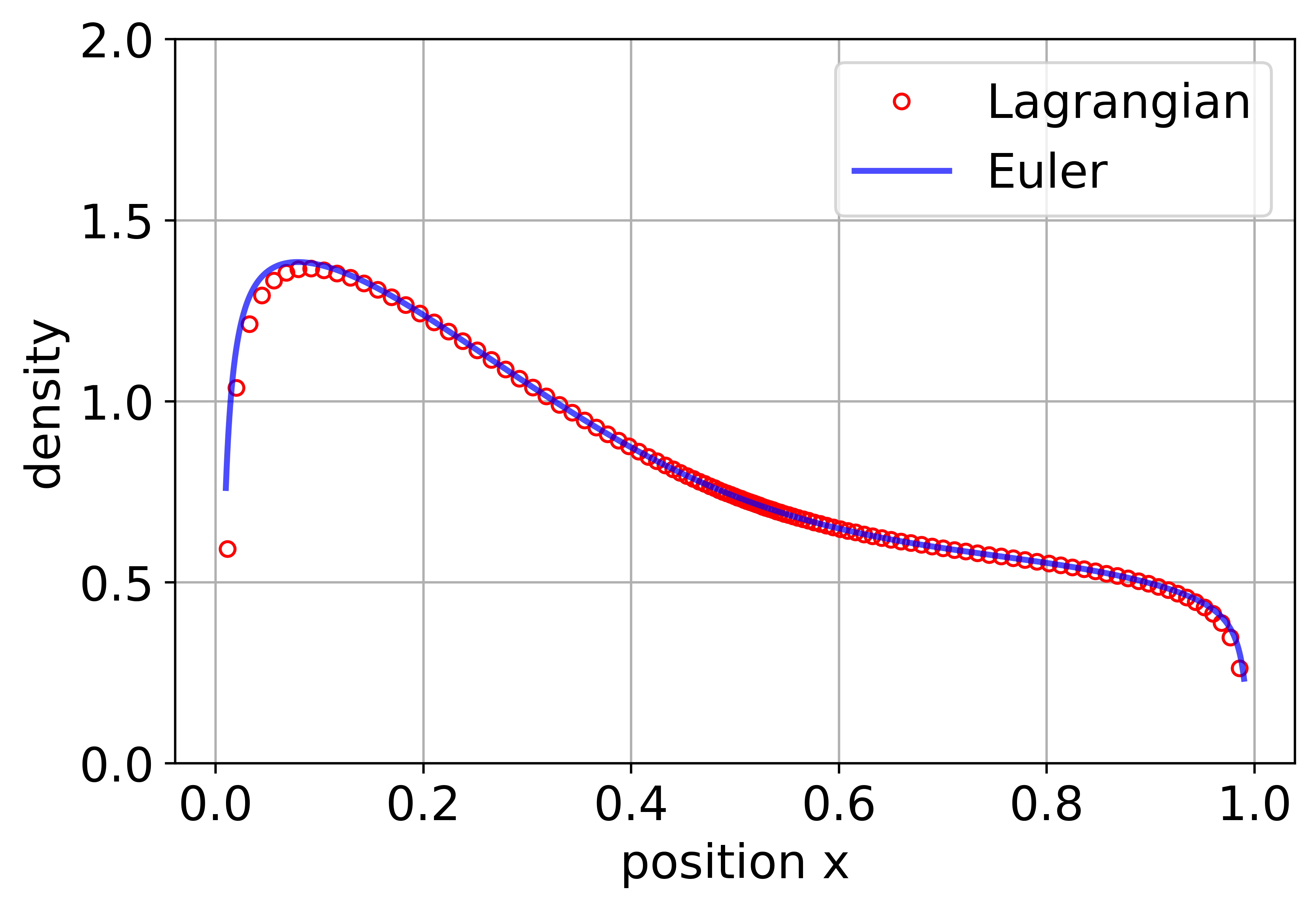}
        \put(-3,55){(d)}
    \end{overpic}
    \hfill
    \begin{overpic}[width = 0.32\textwidth]{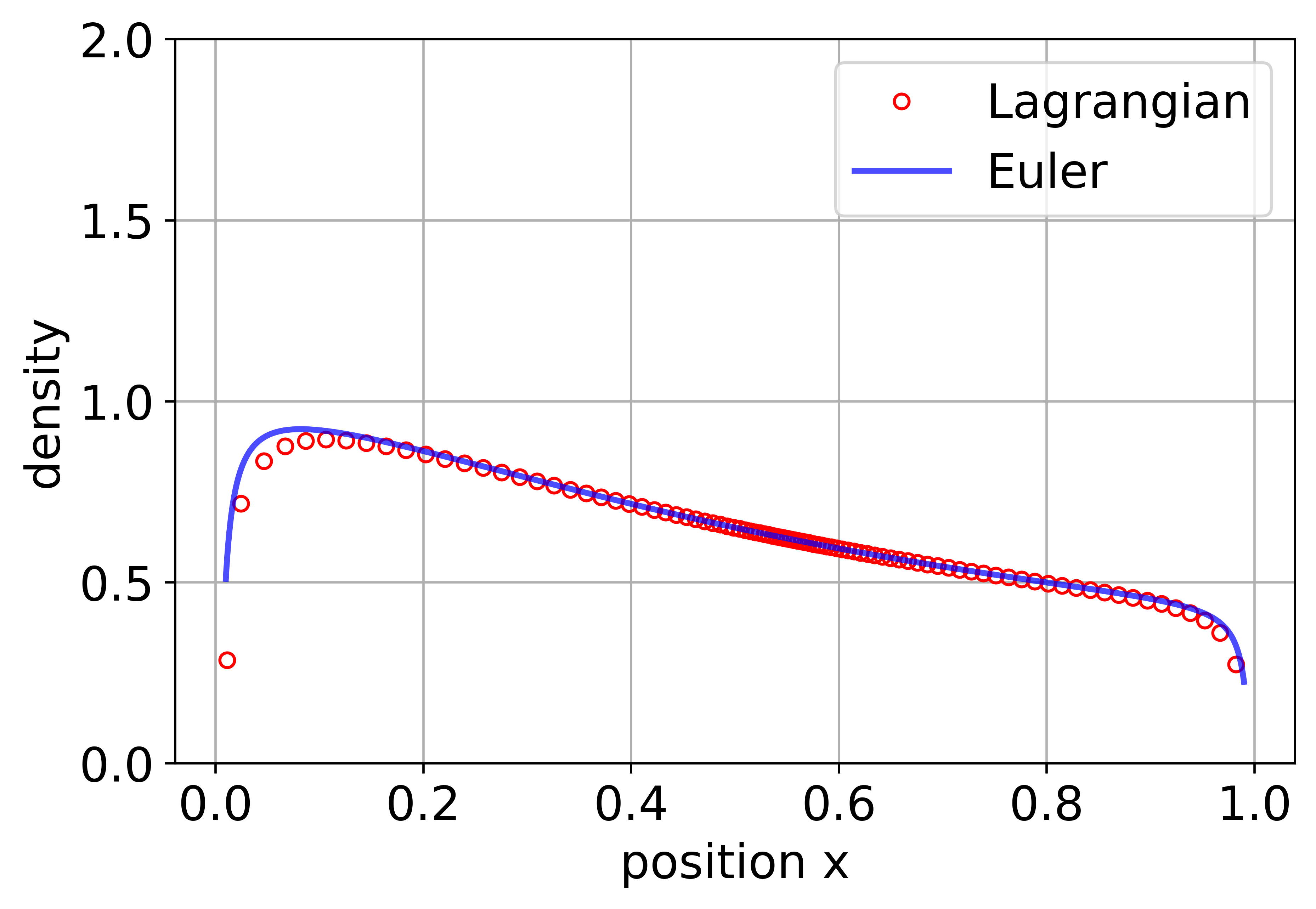}
        \put(-3,55){(e)}
    \end{overpic}
    \hfill
    \begin{overpic}[width = 0.32\textwidth]{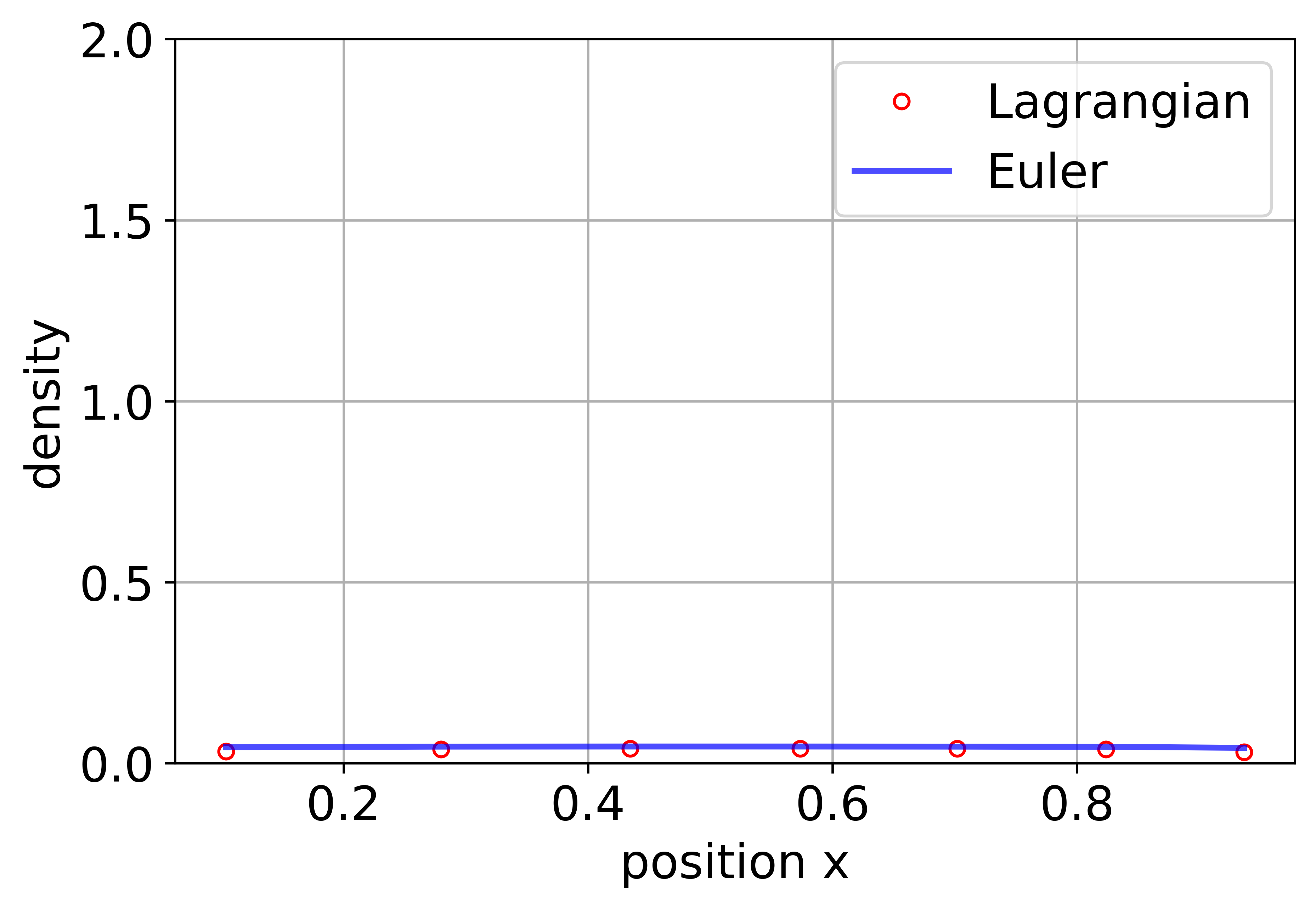}
        \put(-3,55){(f)}
    \end{overpic}  
\setcounter{figure}{0}
\caption{Density evolution for $\rho^{1}_{0}(X)$ at (a) $t = 0.1$, (b)$t = 0.2$, and (c) $t = 1.5$; and for $\rho^{2}_{0}(X)$ at (d) $t = 0.1$, (e) $t = 0.2$, and (f) $t= 1.5$ with $h = \frac{1 - 2\delta}{150}$, $\tau = \frac{1}{10000}$, $\alpha = 2(1-\delta)$, and $N = 150$}
\label{fig:density_f1_alpha_0.7_2(1-delta)}
\end{figure}

We also compare the boundary dynamics, i.e., $a(t)$ and $b(t)$, in the Lagrangian solution with those in the Eulerian solution. The results in  Figure~ \ref{fig:mass_functions_trun_Gaussian} show the evolution of the mass functions $a(t)$ and $b(t)$ with $\delta = 0.01$. It can be seen that for both initial densities, the sum of the mass at $a(t)$ and $b(t)$ approaches $1$, indicating that gene fixation at the boundary is achieved in our modified model. 
\begin{figure}[h]
\centering 
\begin{overpic}[width=0.42 \textwidth]{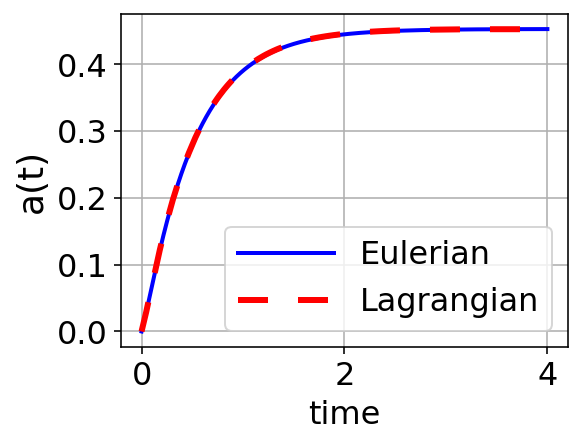}
    \put(-3 ,70){ (a) }
\end{overpic}
\hspace{1em}
\begin{overpic}[width=0.42\textwidth]{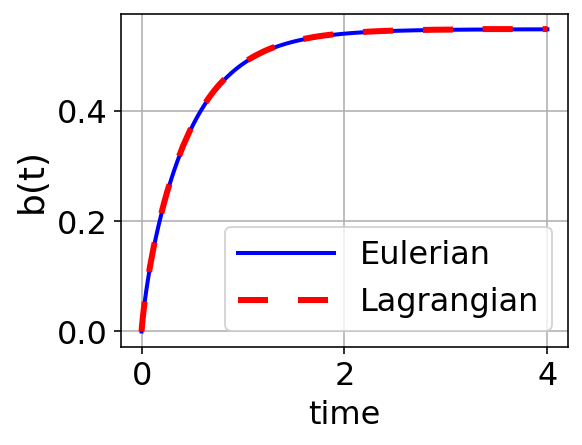}
    \put(-3 ,70){(b)}
\end{overpic}

\vspace{1.5 em}
\begin{overpic}[width = 0.42\textwidth]{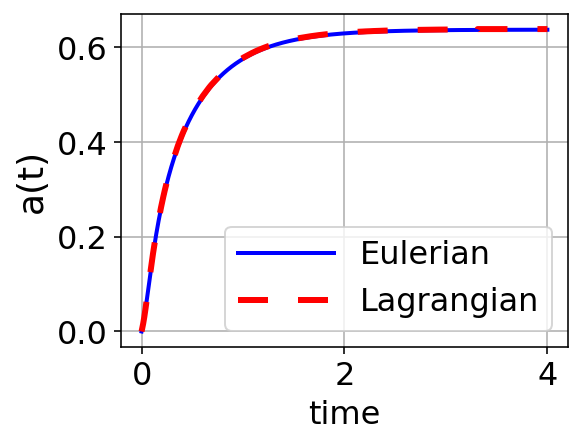}
    \put(-3 ,70){(c)}
\end{overpic}
\hspace{1em}
\begin{overpic}[width = 0.42\textwidth]{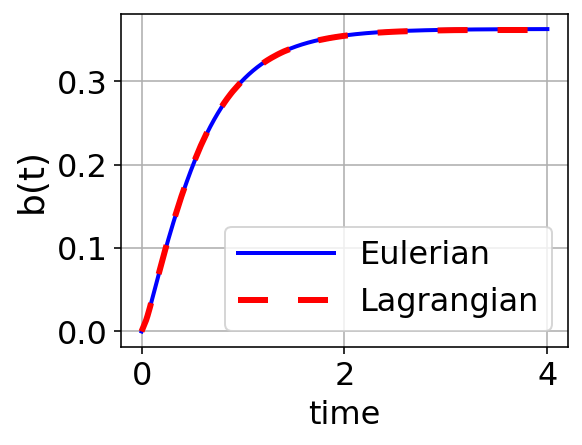}
       \put(-3 ,70){(d)}
    \end{overpic}
\linebreak
\caption{Comparison of the Lagrangian mass functions with the Eulerian solutions: (a) $a(t)$ (upper left) and (b) $b(t)$ (upper right) with initial density $\rho_{0}^{1}$; (c) $a(t)$ (lower left) and (d) $b(t)$ with initial density $\rho^{2}_{0}$. $h = \frac{1 - 2\delta}{1200}$, $\tau = \frac{1}{10000}$, $\alpha = 2(1 - \delta)$, and $N = 1200$.}
\label{fig:mass_functions_trun_Gaussian}
\end{figure}

In addition, we present numerical values in Table \ref{Table1} corresponding to different grid sizes. We define the  $L^{\infty}$ norm for the density on the spatial interval  $(\delta, 1 - \delta)$  at time $t_{n} = T$ as 
\begin{equation}
    \|\rho\|_{L^\infty} = \max_{i^{n}_{c} \leq i \leq i^{n}_{f} - 1}|\rho^{n}_{i}|
\end{equation}, 
and the $L^{\infty}$ norm for the mass functions $a(t)$ and $b(t)$ on the time interval  $(0,T)$ as 
\begin{equation}
    \|a\|_{\infty} = \max_{0\leq j \leq n}|a^{j}|, \quad \|b\|_{\infty} = \max_{0\leq j \leq n}|b^{j}|, 
\end{equation}
where $t_{n} = T$. 
In this table, we set $\delta = 0.01$ and $T = 1.0$. In Table \ref{Table1}, we compare our Lagrangian solutions for the densities with the Eulerian reference solutions. Since the positions of the particles in the Lagrangian scheme change at each time step, we use SciPy’s B-spline interpolation package in the Eulerian scheme to compute numerical errors at the Lagrangian points. The errors in the density indicate that we can accurately capture the solution with a small number of particles. 
\begin{table}[!h]
  \begin{center}
    {\footnotesize
  \begin{tabular}{ c | c  c  c   | c  c  c  }
    \hline
                    &  \multicolumn{3}{c|}{$\rho_0^1(X)$} &  \multicolumn{3}{c}{$\rho_0^2(X)$}  \\ \hline
                   h      & $\|\rho - \tilde{\rho}\|_{\infty}$  &  $\|a(t) - \tilde{a}(t) \|_{\infty}$    & $\|b(t) - \tilde{b}(t) \|_{\infty}$  &  $\|\rho - \tilde{\rho}\|_{\infty}$  &   $\|a(t) - \tilde{a}(t)\|_{\infty}$  &  $\|b(t) - \tilde{b}(t)\|_{\infty}$  \\ \hline
                   $\frac{1-2\delta}{150}$ &  $3.7195\mathrm{e}{-2}$ & $5.4295\mathrm{e}{-3}$ 
                    &  $5.6127\mathrm{e}{-3}$ & $4.7289\mathrm{e}{-2}$ & $1.0775\mathrm{e}{-2}$ & $4.4524 \mathrm{e}{-3}$   \\ \hline 
                   $\frac{1-2\delta}{300}$ &  $3.6905\mathrm{e}{-2}$ & $3.2485\mathrm{e}{-3}$ & $3.4377 \mathrm{e}{-3}$ &  $3.6148\mathrm{e}{-2}$ & $6.5058\mathrm{e}{-3}$ &$2.6633 \mathrm{e}{-3}$ \\ \hline  
                   $\frac{1 - 2\delta}{600}$ & $3.5005\mathrm{e}{-2}$ & $1.9474\mathrm{e}{-3}$ & $2.1885 \mathrm{e}{-3}$ &  $2.3223\mathrm{e}{-2}$ & $1.2063\mathrm{e}{-3}  $ &  $1.2663\mathrm{e}{-3}$ \\ \hline 
                   $\frac{1 - 2\delta}{1200}$ & $2.7280\mathrm{e}{-2}$ & $1.1025\mathrm{e}{-3}$ &$1.5097 \mathrm{e}{-3}$ & $2.2778\mathrm{e}{-2}$ & $7.9115\mathrm{e}{-4}$ &  $8.7346\mathrm{e}{-4}$ \\ \hline 
  \end{tabular}
  }
  \caption{Numerical results of the density values with different spatial grid size and the parameter value $\alpha = 2(1 - \delta)$ to a terminal time $T = 1.0$. The functions $\rho$, $a(t)$, and $b(t)$ represent the density and mass functions obtained from the Lagrangian scheme, while $\tilde{\rho}$, $\tilde{a}(t)$, and $\tilde{b}(t)$ represent the density and mass functions obtained from the Eueler scheme. }
  \label{Table1}
\end{center}
\end{table}

Next, we demonstrate the structure-preserving property of the proposed scheme.
Figure \ref{fig:mass} shows the temporal evolution of the total mass and the deviation of the first moment from its initial value, $|\mathcal{M}_{1}(t) - \mathcal{M}_{1}(0)|$, for both initial densities, with $\alpha = 2(1 - \delta)$ and $N = 1200$. 
The numerical first moment is defined as
\begin{equation}\label{numerical_first_moment}
    \mathcal{M}_{1}^{n} = \frac{\delta}{2} a^{n} + \sum_{i = i^{n}_{c}}^{i^{n}_{f} - 1}\rho^{n}\left(\frac{x^{n}_{i+1} + x^{n}_{i}}{2}\right) \left(\frac{x^{n}_{i+1} + x^{n}_{i}}{2}\right)(x^{n}_{i+1} - x^{n}_{i}) + \left(1 - \frac{\delta}{2}\right)b^{n}
\end{equation}
for each time step $n$. It can be noticed that the numerical solution conserved the total mass and first moment, consistent with the theoretical results.  

Unlike the original model, which has degeneracy at the boundary that complicates the imposition of boundary conditions for first moment conservation, our new model conserves the first moment with well-defined Robin-type boundary conditions. 

\begin{figure}[H]
\centering 
    \begin{overpic}[width = 0.25\textwidth]{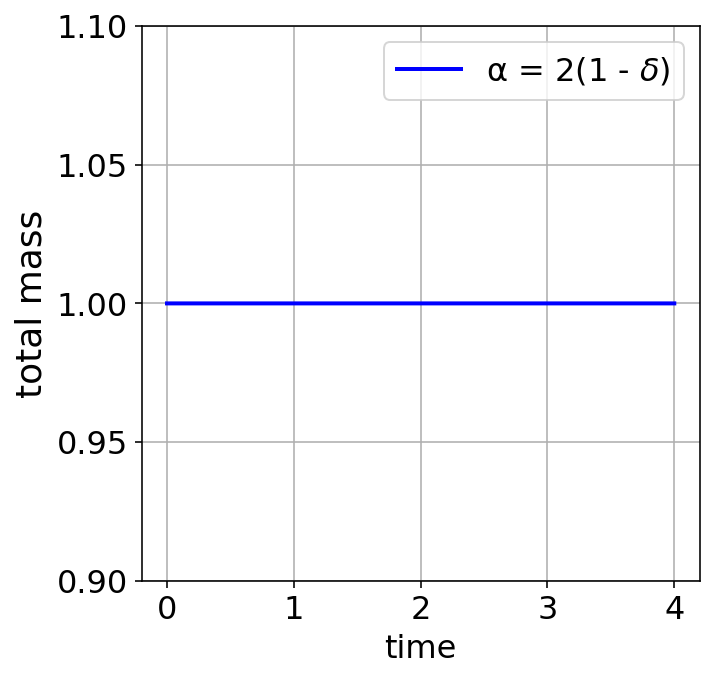}
        \put(-5,85){(a)}
    \end{overpic} 
    \begin{overpic}[width = 0.25\textwidth]{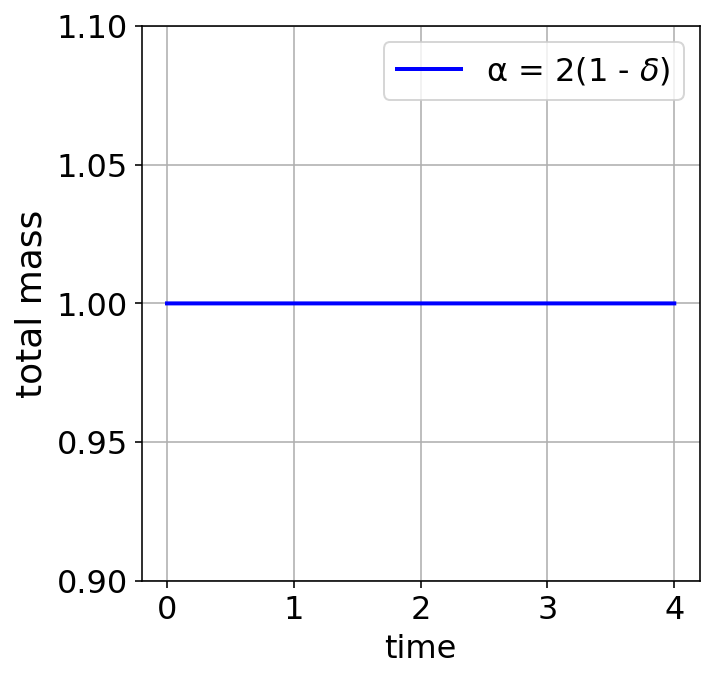} 
        \put(-5,85){(b)}
    \end{overpic}
     \begin{overpic}[width = 0.22\textwidth]{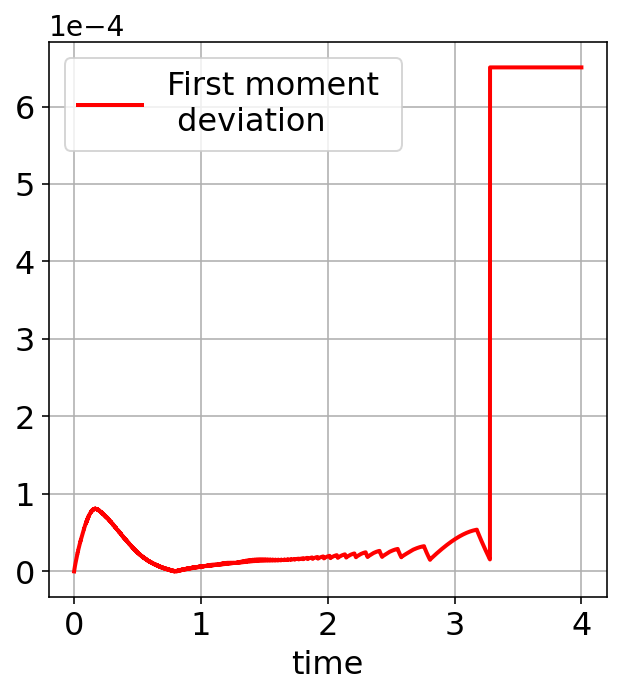}
        \put(-5,87){(c)}
    \end{overpic} 
    \begin{overpic}[width = 0.22\textwidth]{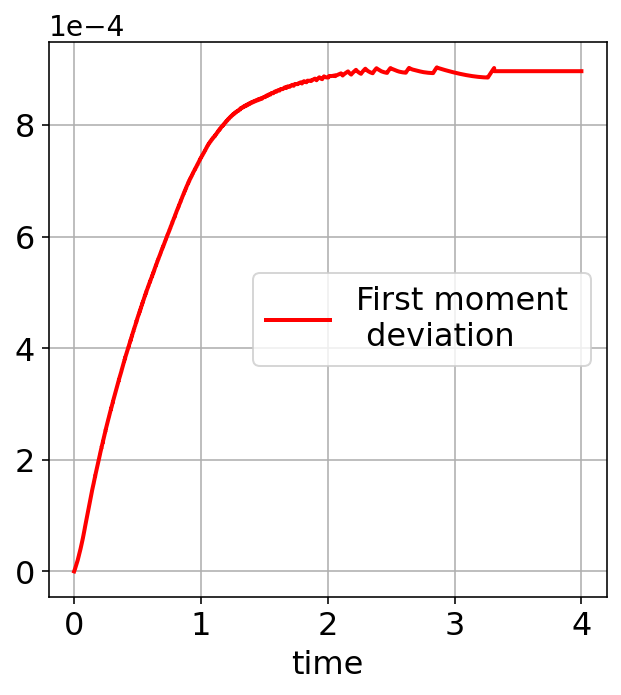} 
        \put(-5,86){(d)}
    \end{overpic}
\linebreak
\caption{Mass and first moment deviation $\mathcal{M}_{1}(t) - \mathcal{M}_{1}(0)$ evolution for $\rho_{0}^{1}$ in (a) and (b), and for $\rho_{0}^{2}$ in (c) and (d), with $\alpha = 2(1 -\delta)$.}
\label{fig:mass}
\end{figure}

We also study the evolution of discrete free energy in the numerical solutions. Initially, the energy decays due to the diffusion of the particles inside the domain. As time progresses, more particles move toward the boundary, causing the energy to increase. Finally, when all the particles are absorbed at the boundary, the energy converges to equilibrium, as shown in the figures.

\begin{figure}[H]
\centering 
    \begin{overpic}[width = 0.35\textwidth]{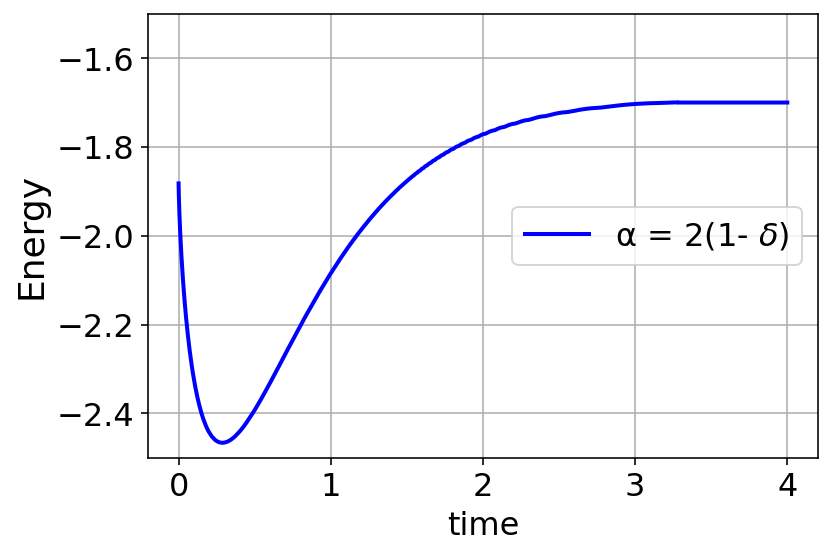}
        \put(-5,60){(a)}
    \end{overpic} 
    \hspace{1.5 em}
    \begin{overpic}[width = 0.35\textwidth]{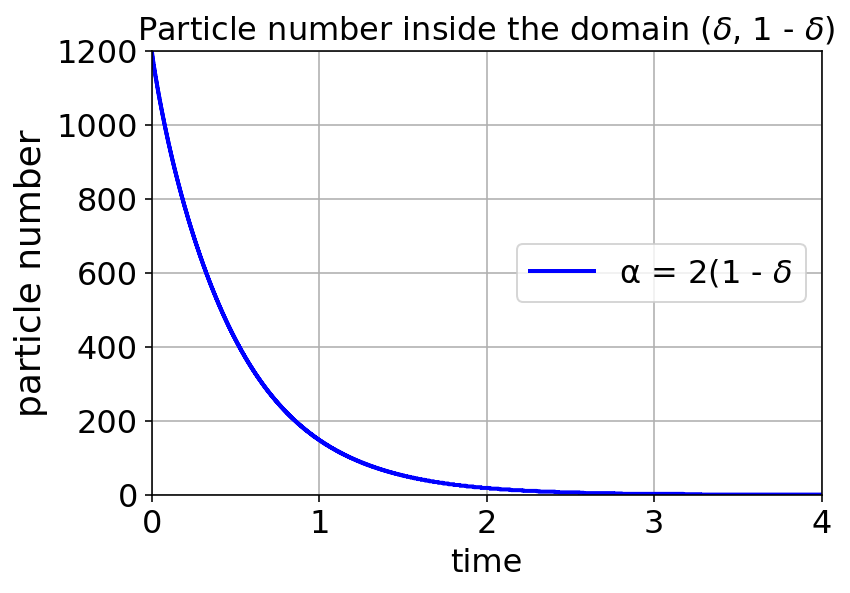}
         \put(-4,60){(b)}
    \end{overpic} 

    \vspace{1em}
    \begin{overpic}[width = 0.35\textwidth]{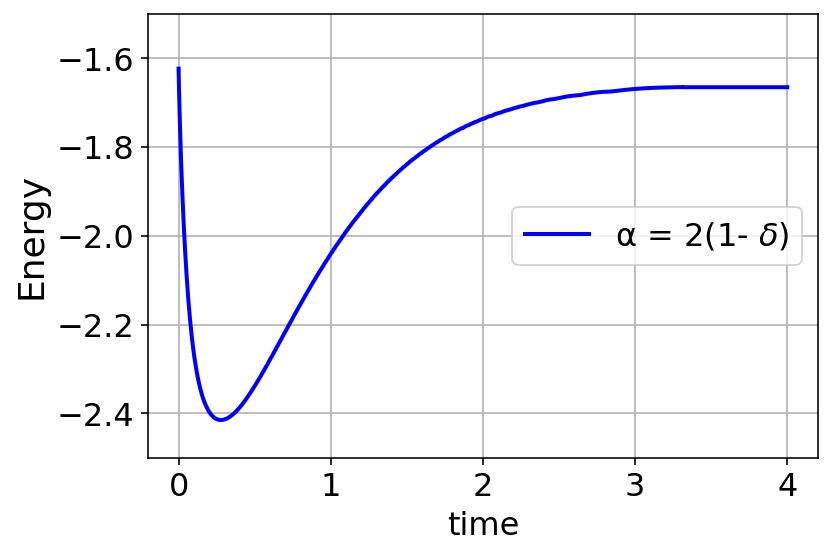}
        \put(-5,60){(c) }
    \end{overpic} 
    \hspace{1.5em}
    \begin{overpic}[width = 0.35\textwidth]{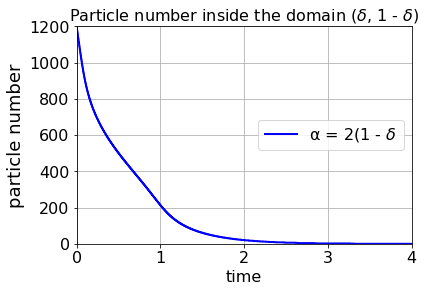}
        \put(-4,60){(d)}
    \end{overpic} 
\linebreak
\caption{Energy and Particle Number evolution for $\rho_{0}^{2}$ with $\alpha = 2(1 -\delta)$. The initial decay of the energy is due to the diffusion of the particles within the domain. When the particles start moving towards the boundary, the energy increases and approaches the steady state, which suggests that all the particles have been absorbed into the boundary.} 
\label{fig:Energy_evolution_trun_Gaussian}
\end{figure}

Finally, we investigate the effects of $\delta$ on the boundary dynamics using the current numerical scheme. We conduct numerical simulations with $\delta = 10^{-2}, 10^{-3}$ and $10^{-4}$. We denote $a_{ref}$ and $b_{ref}$ as the mass functions calculated for $\delta = 10^{-4}$. Other parameters, such as the number of initial particles and temporal step-size are set the same as in Fig. \ref{fig:mass_functions_trun_Gaussian}. Fig. \ref{fig:mass_functions_deviation} shows the  differences between the numerical solutions for $\delta = 10^{-2}$ and $\delta  = 10^{-3}$ compared to $\delta = 10^{-4}$. The numerical results indicate that the differences in boundary dynamics for different values of  $\delta$ are of order $O(\delta)$.  As $\delta$ decreases, these differences diminish and become small relative to the solution with the smallest $\delta$ value. The behavior is evident for moderately small $\delta$, suggesting that the numerical scheme captures the essential dynamics without requiring extremely small $\delta$.
\begin{figure}[H]
\centering 
    \begin{overpic}[width = 0.45\textwidth]{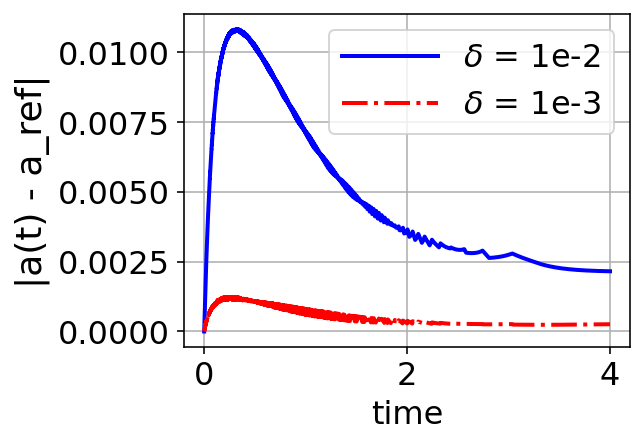}
        \put(-4,65){(a)}
    \end{overpic}
    \hspace{1em}
    \begin{overpic}[width = 0.45\textwidth]{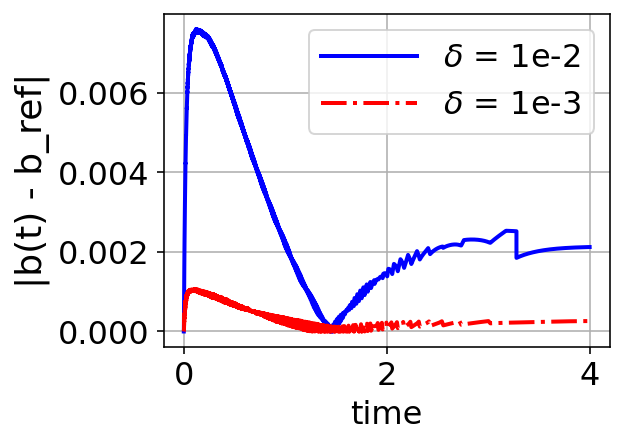}
        \put(-4,65){(b)}
    \end{overpic}

    \begin{overpic}[width = 0.45\textwidth]{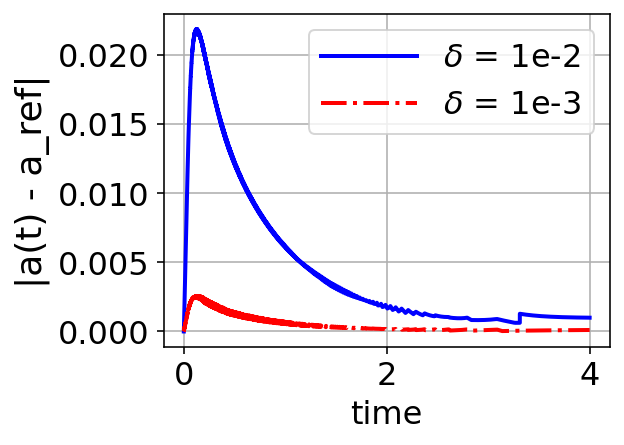}
        \put(-4,65){(c)}
    \end{overpic}
    \hspace{1em}
    \begin{overpic}[width = 0.45\textwidth]{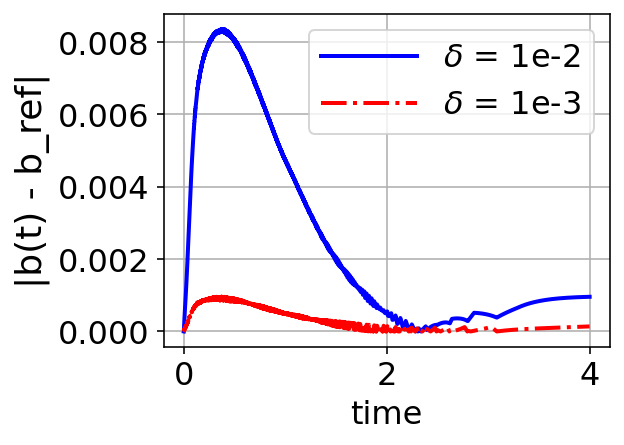}
        \put(-4,65){(d)}
    \end{overpic}
\linebreak
\caption{Evolution of mass functions with $a(t)$ and $b(t)$ with the initial density $\rho_{0}^{1}$, different values of $\delta$, and $\alpha = 2(1-\delta)$. Here, $a_{ref}$ and $b_{ref}$ represent the mass functions with $\delta = 10^{-4}$.}
\label{fig:mass_functions_deviation}
\end{figure}

\section{Conclusion}

In this paper, we propose a modified model that admits classical solutions by changing the domain of the original Kimura equation from $(0,1)$ to $(\delta, 1 - \delta)$ with $\delta$ being a small parameter. This modification allows us to impose a Robin-type boundary condition at $x = \delta$ and $1 - \delta$.
To maintain the biological significance of the model, we introduce two additional variables $a(t)$ and $b(t)$ for the probabilities  in the boundary region to model the behavior of genetic drift near boundaries, which allows us to capture the fixation dynamics.

 To nvestigate the new model numerically, we develop a hybrid Eulerian-Lagrangian operator splitting scheme for the modified random genetic drift model. This scheme first solves the flow map equations \eqref{step1} in the bulk region using a Lagrangian approach, which tracks individual particles while enforcing a no-flux boundary condition. The boundary dynamics are then handled in Eulerian coordinates, providing a framework for managing particle interactions near the boundaries. This hybrid scheme guarantees mass conservation, maintains positivity, and preserves the first moment.
The numerical tests conducted highlight the efficiency, accuracy, and structure-preserving properties of the proposed scheme, demonstrating its ability to capture essential features of the model. 

Despite these advancements, several challenges remain. Extending our approach to higher-dimensional problems, such as those involving multiple alleles, is nontrivial. Additionally, we do not provide a rigorous proof of convergence or error estimates, and our numerical results do not yield clear convergence rates. Future work will focus on developing a more accurate numerical scheme for higher-dimensional settings and establishing a rigorous framework for analyzing the convergence and error behavior of our method. 

\section*{Acknowledgment} C. Liu and C. Chen were partially supported by NSF grants DMS-2118181 and DMS-2410742. Y. Wang was partially supported by NSF grant DMS-2410740. \\ 

\appendix 
\section*{Appendix}
In the appendix, we give the details of the derivation of the Kimura equation. The most of the material here is based on \cite{kimura1964diffusion}.  \\
\indent Consider a population of size $N$ that contains a pair of alleles, $A_{1}$ and $A_{2}$.  
We assume that the change of gene frequencies between generations follows a Markovian process , which means that the probability distribution of a gene frequency in the future of only depends on the present state and not the past history. We also assume that the population size remains the same at each generation, which implies that the number of genes also remains $2N$. Let the gene frequencies of $A_{1}$ and $A_{2}$ be $x$ and $1 - x$, respectively and denote $\phi(p,x;t)$ as the conditional probability density that the gene frequency of $A_{1}$ is $x$ at time t given that its initial proportion is $p$.  
With the total number of genes being $2N$, the frequency distribution can be written as 
\begin{equation}
    \rho(x,t) = \frac{\phi(p,x;,t)}{2N}
\end{equation}
Now, let $g(\Delta x, x; \Delta t, t)$ be the probability density function for the change in gene frequency from $x$ to $x + \Delta x$ over the time interval $(t, t + \Delta t)$.  Then under the assumption that the process is Markovian, we have
\begin{equation}\label{change_in_frequency}
    \phi(p,x;t+ \Delta t) = \int \phi(p,x - \Delta x;t)g(\Delta x, x - \Delta x; \Delta t,t)d(\Delta x),
\end{equation}
where the integral on the right is taken over all possible values of $\Delta x$ such that $x - \Delta x$ lies within the interval $[0,1]$. Provided that both $\phi(p,x;t)$ and $g(\Delta x, x; \Delta t,t)$ are smooth functions with respect to the variables $x$ and $t$, we may apply the Taylor expansion of the integrand on the right-hand side of \eqref{change_in_frequency} in terms of $\Delta x$, and obtain 

\begin{equation}\label{Taylor_expansions}
\begin{split}
    & \phi(p,x - \Delta x;t)g(\Delta x, x - \Delta x;\Delta t, t) \\ 
    & = \phi(p,x;t)g(\Delta x,x; \Delta t, t) - (\Delta x)  \frac{\partial }{\partial x}\left[\phi(p,x;t)g(\Delta x,x ;\Delta t, t)\right] + \frac{(\Delta x)^{2}}{2!}\frac{\partial^{2}}{\partial x^{2}}  \left[\phi(p,x;t)g(\Delta x, x; \Delta t, t)\right] \\
    & + \cdots + \frac{(\Delta x)^n}{n!} \frac{\partial^n}{\partial x^n} \left[ \phi(p,x;t) g(\Delta x, x; \Delta t, t) \right] + R_{n}(\Delta x, x),
\end{split}
\end{equation}
where $\xi \in (x - \Delta x, x)$.
By plugging \eqref{Taylor_expansions} into \eqref{change_in_frequency}, we can obtain the following approximation 
\begin{equation} \label{Approximation}
\begin{split}
    \phi(p,x; t + \Delta t) & = \phi(p,x;t) \int g(\Delta x, x; \Delta t, t) d(\Delta x) \\
                            & - \frac{\partial }{\partial x} \left[\phi(p,x;t)\int \Delta x g(\Delta x, x;\Delta t, t) d(\Delta x)\right]  \\ 
                            & + \frac{1}{2}\frac{\partial^{2}}{\partial x^{2}}\left[\phi(p,x;t)\int (\Delta x)^{2} g(\Delta x, x; \Delta t, t) d(\Delta x)\right] \\ 
                            & + \cdots + \frac{1}{n!}\frac{\partial^{n}}{\partial x^{n}}\left[\phi(p,x;t)\int (\Delta x)^{n} g(\Delta x, x; \Delta t, t) d(\Delta x)\right] + \tilde{R}_{n}(\Delta x, x). 
\end{split}
\end{equation}
Since $g$ is a probability density, we have 
\begin{equation*}
    \int g d(\Delta x) = 1. 
\end{equation*}
Then we can move the first term on the right-hand side of \eqref{Approximation} to the left and divide both sides by $\Delta t$ to get    
\begin{equation}
\begin{split}
    \frac{\phi(p,x;t + \Delta t) - \phi(p,x;t)}{\Delta t} &= - \frac{\partial}{\partial x}\left[\phi(p,x;t)\frac{1}{\Delta t} \int (\Delta x) g(\Delta x,x;\Delta t,t) d(\Delta x) \right]  \\
    & + \frac{1}{2} \frac{\partial^{2}}{\partial x^{2}} \left[\phi(p,x;t) \frac{1}{\Delta t} \int (\Delta x)^{2} g(\Delta x, x; \Delta t, t) d(\Delta x)\right] \\
    & + \cdots + \frac{1}{n!}\frac{\partial^{n}}{\partial x^{n}}\left[\phi(p,x;t) \frac{1}{\Delta t}\int (\Delta x)^{n} g(\Delta x, x; \Delta t, t) d(\Delta x)\right] + \frac{1}{\Delta t }\tilde{R}_{n}(\Delta x, x).
\end{split}
\end{equation}
By taking the limit as $\Delta t$ goes to zero, and let 
\begin{equation}
    \lim_{\Delta t \rightarrow 0} \frac{1}{\Delta t} \int (\Delta x) g(\Delta x,x;\Delta t, t) d(\Delta x) = M(x,t), 
\end{equation}
\begin{equation}
    \lim_{\Delta t \rightarrow 0} \frac{1}{\Delta t} \int (\Delta x)^{2} g(\Delta x,x;\Delta t, t) d(\Delta x) = V(x,t), 
\end{equation}
where $M(x,t)$ and $V(x,t)$ stands for the first and the second moments of $\Delta x$ 
over the infinitesimal time interval $(t, t + \Delta t)$. 
Finally, under the assumption that 
\begin{equation}
    \lim_{\Delta \rightarrow 0} \frac{1}{\Delta t} \int (\Delta x)^{n} g(\Delta x, x; \Delta t, t) d(\Delta x) = 0 
\end{equation} for $n \geq 3$, we arrive at the Fokker-Planck equation: 
\begin{equation}\label{PDE}
    \frac{\partial \phi(p,x;t)}{\partial t} = \frac{1}{2} \frac{\partial^{2}}{\partial x^{2}}\left[ V(x,t) \phi(p,x;t)\right] - \frac{\partial}{\partial x} \left[M(x,t)\phi(p,x;t)\right]. 
\end{equation}  
Since data such as mutation rates, migration rates, and selection coefficients can only be measured at each generation, $M(x,t)$ and $V(x,t)$ are usually assumed to depend solely on the gene frequency $x$. \\ 
\indent Now, in the pure random drift case, the first and second moment $M$ and $V$ are chosen to be 
zero and $x(1-x)/2N$, respectively. Hence, by plugging the expressions of $M$ and $V$ into \eqref{PDE}, we obtain
\begin{equation}
    \frac{\partial \phi}{\partial t} = \frac{1}{4N} \frac{\partial^{2}}{\partial x^{2}}\left[x(1-x)\phi \right], 0 < x < 1.
\end{equation}


\begin{thebibliography}{10}

\bibitem{arnol2013mathematical}
Vladimir~Igorevich Arnol'd.
\newblock {\em Mathematical methods of classical mechanics}, volume~60.
\newblock Springer Science \& Business Media, 2013.

\bibitem{barzilai1988two}
Jonathan Barzilai and Jonathan~M Borwein.
\newblock Two-point step size gradient methods.
\newblock {\em IMA journal of numerical analysis}, 8(1):141--148, 1988.

\bibitem{carrillo2022optimal}
Jos{\'e}~A Carrillo, Lin Chen, and Qi~Wang.
\newblock An optimal mass transport method for random genetic drift.
\newblock {\em SIAM Journal on Numerical Analysis}, 60(3):940--969, 2022.

\bibitem{carrillo2010numerical}
Jos{\'e}~A Carrillo and J~Salvador Moll.
\newblock Numerical simulation of diffusive and aggregation phenomena in nonlinear continuity equations by evolving diffeomorphisms.
\newblock {\em SIAM Journal on Scientific Computing}, 31(6):4305--4329, 2010.

\bibitem{chalub2009non}
Fabio~ACC Chalub and Max~O Souza.
\newblock A non-standard evolution problem arising in population genetics.
\newblock 2009.

\bibitem{dangerfield2012boundary}
Ciara~E Dangerfield, David Kay, Shev Macnamara, and Kevin Burrage.
\newblock A boundary preserving numerical algorithm for the wright-fisher model with mutation.
\newblock {\em BIT Numerical Mathematics}, 52:283--304, 2012.

\bibitem{de2013non}
Sybren~Ruurds De~Groot and Peter Mazur.
\newblock {\em Non-equilibrium thermodynamics}.
\newblock Courier Corporation, 2013.

\bibitem{duan2019numerical}
Chenghua Duan, Chun Liu, Cheng Wang, and Xingye Yue.
\newblock Numerical complete solution for random genetic drift by energetic variational approach.
\newblock {\em ESAIM: Mathematical Modelling and Numerical Analysis}, 53(2):615--634, 2019.

\bibitem{epstein2010wright}
Charles~L Epstein and Rafe Mazzeo.
\newblock Wright--fisher diffusion in one dimension.
\newblock {\em SIAM journal on mathematical analysis}, 42(2):568--608, 2010.

\bibitem{epstein2013degenerate}
Charles~L Epstein and Rafe Mazzeo.
\newblock {\em Degenerate diffusion operators arising in population biology}.
\newblock Number 185. Princeton University Press, 2013.

\bibitem{ewens2004mathematical}
Warren~John Ewens.
\newblock {\em Mathematical population genetics: theoretical introduction}, volume~27.
\newblock Springer, 2004.

\bibitem{feller1951diffusion}
William Feller et~al.
\newblock Diffusion processes in genetics.
\newblock 1951.

\bibitem{fisher1923xxi}
Ronald~A Fisher.
\newblock Xxi.—on the dominance ratio.
\newblock {\em Proceedings of the royal society of Edinburgh}, 42:321--341, 1923.

\bibitem{giga2017variational}
Mi-Ho Giga, Arkadz Kirshtein, and Chun Liu.
\newblock Variational modeling and complex fluids.
\newblock {\em Handbook of mathematical analysis in mechanics of viscous fluids}, pages 1--41, 2017.

\bibitem{jenkins2017exact}
Paul~A Jenkins and Dario Spano.
\newblock Exact simulation of the wright--fisher diffusion.
\newblock 2017.

\bibitem{kimura1964diffusion}
Motoo Kimura.
\newblock Diffusion models in population genetics.
\newblock {\em Journal of Applied Probability}, 1(2):177--232, 1964.

\bibitem{kimura1954stochastic}
Motoo Kimura et~al.
\newblock {\em Stochastic processes and distribution of gene frequencies under natural selection}.
\newblock Citeseer, 1954.

\bibitem{knopf2021phase}
Patrik Knopf, Kei~Fong Lam, Chun Liu, and Stefan Metzger.
\newblock Phase-field dynamics with transfer of materials: the cahn--hilliard equation with reaction rate dependent dynamic boundary conditions.
\newblock {\em ESAIM: Mathematical Modelling and Numerical Analysis}, 55(1):229--282, 2021.

\bibitem{liu2023continuum}
Chun Liu, Jan-Eric Sulzbach, and Yiwei Wang.
\newblock On a continuum model for random genetic drift: A dynamical boundary condition approach, 2023.

\bibitem{liu2021structure}
Chun Liu, Cheng Wang, and Yiwei Wang.
\newblock A structure-preserving, operator splitting scheme for reaction-diffusion equations with detailed balance.
\newblock {\em Journal of Computational Physics}, 436:110253, 2021.

\bibitem{liu2020lagrangian}
Chun Liu and Yiwei Wang.
\newblock On lagrangian schemes for porous medium type generalized diffusion equations: A discrete energetic variational approach.
\newblock {\em Journal of Computational Physics}, 417:109566, 2020.

\bibitem{mckane2007singular}
Alan~J McKane and David Waxman.
\newblock Singular solutions of the diffusion equation of population genetics.
\newblock {\em Journal of theoretical biology}, 247(4):849--858, 2007.

\bibitem{onsager1931reciprocal}
Lars Onsager.
\newblock Reciprocal relations in irreversible processes. i.
\newblock {\em Physical review}, 37(4):405, 1931.

\bibitem{onsager1931reciprocal2}
Lars Onsager.
\newblock Reciprocal relations in irreversible processes. ii.
\newblock {\em Physical review}, 38(12):2265, 1931.

\bibitem{strutt1871some}
JW~Strutt.
\newblock Some general theorems relating to vibrations.
\newblock {\em Proceedings of the London Mathematical Society}, 1(1):357--368, 1871.

\bibitem{wang2022some}
Yiwei Wang and Chun Liu.
\newblock Some recent advances in energetic variational approaches.
\newblock {\em Entropy}, 24(5):721, 2022.

\bibitem{wang2020field}
Yiwei Wang, Chun Liu, Pei Liu, and Bob Eisenberg.
\newblock Field theory of reaction-diffusion: Law of mass action with an energetic variational approach.
\newblock {\em Physical Review E}, 102(6):062147, 2020.

\bibitem{westdickenberg2010variational}
Michael Westdickenberg and Jon Wilkening.
\newblock Variational particle schemes for the porous medium equation and for the system of isentropic euler equations.
\newblock {\em ESAIM: Mathematical Modelling and Numerical Analysis}, 44(1):133--166, 2010.

\bibitem{wright1929evolution}
Sewall Wright.
\newblock The evolution of dominance.
\newblock {\em The American Naturalist}, 63(689):556--561, 1929.

\bibitem{wright1937distribution}
Sewall Wright.
\newblock The distribution of gene frequencies in populations.
\newblock {\em Proceedings of the National Academy of Sciences}, 23(6):307--320, 1937.

\bibitem{wright1945differential}
Sewall Wright.
\newblock The differential equation of the distribution of gene frequencies.
\newblock {\em Proceedings of the National Academy of Sciences}, 31(12):382--389, 1945.

\bibitem{xu2019behavior}
Shixin Xu, Minxin Chen, Chun Liu, Ran Zhang, and Xingye Yue.
\newblock Behavior of different numerical schemes for random genetic drift.
\newblock {\em BIT Numerical Mathematics}, 59:797--821, 2019.

\bibitem{zhao2013complete}
Lei Zhao, Xingye Yue, and David Waxman.
\newblock Complete numerical solution of the diffusion equation of random genetic drift.
\newblock {\em Genetics}, 194(4):973--985, 2013.

\end{thebibliography}

\end{document}